\documentclass[12pt]{amsart}
\usepackage{a4wide}

\usepackage{verbatim}
\usepackage[margin=2cm]{geometry}
\usepackage{amsmath}
\usepackage{amsfonts}
\usepackage{amssymb}
\usepackage{amsthm}
\usepackage{mathtools}
\usepackage{url}
\usepackage{enumitem}
\usepackage[dvipsnames]{xcolor}
\usepackage{pgfplots}
\pgfplotsset{compat=1.15}
\usepackage{mathrsfs}
\usetikzlibrary{arrows}
\usepackage{graphicx}
\usepackage[hypertexnames=false]{hyperref}

\newcommand{\RR}{\mathbb{R}}

\newcommand{\HH}{{\mathbb{H}^N}}
\newcommand{\GG}{\Gamma}
\newcommand{\eps}{\varepsilon}

\numberwithin{equation}{section}

\newtheorem{theorem}{Theorem}[section]
\newtheorem{proposition}[theorem]{Proposition}
\newtheorem{lemma}[theorem]{Lemma}

\newtheorem{corollary}[theorem]{Corollary}
\newtheorem{hyp}[theorem]{Hypothesis}
\theoremstyle{definition}
\newtheorem{definition}[theorem]{Definition}

\theoremstyle{remark}
\newtheorem{remark}[theorem]{Remark}

\theoremstyle{remark}
\DeclareRobustCommand{\rchi}{{\mathpalette\irchi\relax}}
\newcommand{\irchi}[2]{\raisebox{\depth}{$#1\chi$}}
\usepackage[toc,page]{appendix}

\newcommand{\dv}{\mathrm{div}}

\newcommand{\dd}{\mathrm{d}}
\newcommand{\dmu}{\dd\mu}
\newcommand{\Vol}{\mathrm{Vol}}
\newcommand{\loc}{\mathrm{loc}}

\def\?[#1]{\textbf{[#1]}\marginpar{\Large{\textbf{??}}}}

\begin{document}
\title[Concentration limit for convection-diffusion on the hyperbolic space]{Concentration limit for non-local dissipative convection-diffusion kernels on the hyperbolic space}

\author{Mar\'ia del Mar Gonz\'alez}
\address[M.~d.~M.~Gonz\'alez]{Departamento de Matematicas, Universidad Aut\'onoma de Madrid and ICMAT. Madrid 28049, Spain}
\email{mariamar.gonzalezn@uam.es}

\author[L.~I.~Ignat]{Liviu I. Ignat}
\address[L.~I.~Ignat]{Institute of Mathematics ``Simion Stoilow'' of the Romanian Academy\\
21 Calea Grivitei Street, 010702 Bucharest, Romania\\
and\\
The Research Institute of the University of Bucharest - ICUB\\
University of Bucharest\\
90-92 Sos. Panduri, 5th District, Bucharest, Romania}
\email{liviu.ignat@gmail.com}

\author[D.~Manea]{Drago\c s Manea}
\address[D.~Manea]{Institute of Mathematics ``Simion Stoilow'' of the Romanian Academy\\
21 Calea Grivitei Street, 010702 Bucharest, Romania\\
and\\
The Research Institute of the University of Bucharest - ICUB\\
University of Bucharest\\
90-92 Sos. Panduri, 5th District, Bucharest, Romania}
\email{dmanea28@gmail.com}

\author[S.~Moroianu]{Sergiu Moroianu}
\address[S.~Moroianu]{Institute of Mathematics ``Simion Stoilow'' of the Romanian Academy\\
21 Calea Grivitei, 010702 Bucharest, Romania\\ and\\
Facultatea de Matematic\u a, strada Academiei 14, Universitatea Bucure\c sti, Rom\^ ania}
\email{moroianu@alum.mit.edu}
\date{\today}
\begin{abstract}
We study a non-local evolution equation on the hyperbolic space $\mathbb{H}^N$.
We first consider a model for particle transport governed by a non-local interaction kernel defined on the tangent bundle and invariant under the geodesic flow. We study the relaxation limit of this model to a local transport problem, as the kernel gets concentrated near the origin of each tangent space. Under some regularity and integrability conditions on the kernel, we prove that the solution of the rescaled non-local problem converges to that of the local transport equation. Then, we construct a large class of interaction kernels that satisfy those conditions.

We also consider a non-local, non-linear convection-diffusion equation on $\mathbb{H}^N$ governed by two kernels, one for each of the diffusion and convection parts, and we prove that the solution converges to the solution of a local problem as the kernels get concentrated. We prove and then use in this sense a compactness tool on manifolds inspired by the work of Bourgain-Brezis-Mironescu.

\end{abstract}
\maketitle
\section{Introduction}
This paper is concerned with non-local time-dependent interaction models for particles in the hyperbolic space $\HH$, 
expressing a mixed non-local diffusion-convection behaviour:
\begin{equation}
\label{IntroNonlocalConvDiff}
\begin{cases}\displaystyle
\begin{aligned}
 \partial_t u(t,x) = &\int_\HH J(d(x,y))(u(t,y)-u(t,x)) \dmu(y)\\
    &+ \int_{\HH} G(x,y)(f(u(t,y))-f(u(t,x))) \dmu(y) \end{aligned} & x\in \HH, t\geq 0;\\
u(0,x)= u_0(x), & x\in \HH.
\end{cases}
\end{equation}
The kernels $J:[0,\infty)\to [0,\infty)$ and $G:\HH\times \HH\to [0,\infty)$ encode the strength of the interaction (non-local diffusion and convection, respectively) between particles at positions $x$ and $y$. The non-linearity $f$ is a locally Lipschitz function, which will eventually be assumed to be of the form $f(r)= |r|^{q-1} r$, $q\geq 1$.

We aim to analyze the behavior of the solutions of \eqref{IntroNonlocalConvDiff} for a family of inhomogeneously rescaled kernels $(J_\eps)_{\eps>0}$ and $(G_\eps)_{\eps>0}$, describing, in the limit $\eps \to 0$, a concentration of the interactions to small distances between particles. The diffusion kernel $J$, which depends only on the distance between pairs of points, can be easily defined and rescaled on every complete Riemannian manifold. One of the  challenges  we faced was to properly define on $\HH$ a non-local convection kernel admitting a meaningful rescaling and dissipating the $L^2$ energy uniformly.

There exists a vast literature regarding evolution equations on manifolds and, in particular, on the hyperbolic space. For instance, the heat kernel on $\HH$ was computed explicitly and estimated uniformly in space and time by Davies and Mandouvalos \cite{Davies_Mandouvalos_1988} (see also \cite{Grigoryan_Noguchi_1998}). The existence of asymptotic profiles for the heat equation was studied by V\' azquez \cite{VazquezHeatHyperbolic}, in the case of the hyperbolic space and by Anker, Papageorgiou and Zhang \cite{Anker_Papageorgiou_Zhang}, in the more general case of symmetric spaces of noncompact type.

The existence of solutions for a non-linear heat equation on $\HH$ was discussed by Bandle, Pozio and Tesei \cite{Bandle_Pozio_Tesei_2011}, whereas Banica, Gonz\' alez and S\'{a}ez \cite{MR3375868} studied an extension problem for the fractional
Laplacian on non-compact manifolds.
The wave equation and corresponding Strichartz estimates on $\HH$ were analysed in \cite{tataru} and \cite{Anker_Pierfelice_Vallarino_2012}. The linear and non-linear Schr\" odinger equations on the hyperbolic space were studied by Banica \cite{Banica_2007}
and Anker and Pierfelice \cite{Anker_Pierfelice_2009}, whereas the incompressible Navier–Stokes system was considered by Pierfelice in \cite{Pierfelice_2017}. 
Many of the results above rely on the so-called Fourier-Helgason transform (see \cite{Helgason_1984} and \cite{brayHarmonicHyperbolic}), which can be used to simplify equations involving the Laplace-Beltrami operator on $\HH$. However, this transform does not behave well with respect to first-order differential operators such as the Riemannian gradient, hence, in the present paper, we preferred to use a different approach to non-local and local convection-diffusion problems.

There exists also an extensive literature on non-local diffusion-convection phenomena in Euclidean space, arising for instance from the peridynamic theory of continuous mechanics \cite{MR3560288} or from model processes, for example, in biology, dislocations dynamics etc.\ \cite{MR3938295}. 
From a mathematical point of view, there exist two types of non-local interaction operators that may be considered, depending on the properties of their kernels. The first one consists of integrable (or even smooth, compactly supported) kernels \cite{rossiNonLocal}, whereas the second type is represented by singular kernels similar to the fractional Laplacian \cite{MR3469920,MR3445279}.

Our paper deals with the first type of kernels \cite{rossiNonLocal} for which the well-posedness and convergence of solution to local problems have been previously studied in the Euclidean setting \cite{Ignat-Rossi-nonlocal, IgnatIgnatStancuDumitru}, but also in the case of non-local diffusion on quantum graphs \cite{IgnatRossiAngel}. 
In the Euclidean case there are various ways to model the convection by non-local operators \cite{Ignat-Rossi-nonlocal, MR2888353} and their numerical approximations  \cite{MR3170190, MR3646345}.
In \cite{Ignat-Rossi-nonlocal, IgnatIgnatStancuDumitru} the family of rescaled equations has the following form: 
\begin{equation}
\label{IntrotranspNonLocRNEps}
\begin{cases}\displaystyle
\begin{aligned}
\partial_t u(t,x) = {}&\eps^{-N-2}\int_{\RR^N}J\left(\frac{|y-x|}{\eps}\right)(u(t,y)-u(t,x)) \dd y \\
&+ \eps^{-N-1}\int_{\RR^N}\widetilde{G}\left(\frac{y-x}{\eps}\right)(f(u(t,y))-f(u(t,x))) \dd y,
\end{aligned} &\quad x\in \RR^N, t\geq 0;\\[10pt]
 u(0,x) = u_0(x), & x\in \RR^N.
\end{cases}
\end{equation}
for the point-pair invariant kernel $J:[0,\infty)\to [0,\infty)$ and the convolution kernel $\widetilde{G}:\RR^N \to [0,\infty)$.

As $\eps \to 0$, the solutions of \eqref{IntrotranspNonLocRNEps} converge strongly  in $L^2((0,T)\times\RR^N)$ (see   \cite[Th.~1.2]{Ignat-Rossi-nonlocal} for kernels $J$ and $G$ in $\mathcal{S}(\RR^N)$ and \cite{IgnatIgnatStancuDumitru} for  more general kernels) to the solution of the local convection-diffusion problem:
\begin{equation}
\label{IntrotranspLocRNEps}
\begin{cases}
\partial_t u(t,x) = A_J \Delta u(t,x) - \dv(f(u(t))X_G)(x), &  x\in \RR^N, t\geq 0;\\
u(0,x) =u_0(x),&  x\in \RR^N,
\end{cases}
\end{equation}
where 
\begin{equation}
    \label{AJ}
    A_J=\frac{1}{2N}\int_{\RR^N} J(|x|_e)|x|_e^2 \dd x>0
\end{equation}
 and $X_G=-\int_{\RR^N} \widetilde{G}(x) x\hspace{0.1cm} \dd x\in \RR^N$. Throughout the paper, $|x|_e$ stands for the Euclidean norm of the vector $x\in \RR^N$.

\subsection*{Convergence results in non-zero curvature}
Our goal is to study this concentration phenomenon outside the flat Euclidean setting, and to check whether the negative curvature has a qualitative effect on the convergence results stated above. We focus on the hyperbolic space $\HH$, arguably the simplest non-flat, complete non-compact Riemannian manifold. As a manifold, $\HH$ is diffeomorphic to $\RR^N$, and it has constant sectional curvature $-1$. 

A non-local diffusion model on $\HH$ has been already studied in a recent paper \cite{bandle2018}. The authors proved that the solutions of the evolution equation
\begin{equation}
\label{bandle.NonlocalDiff}
\begin{cases}
\partial_t u(t,x)=\int_\HH J(d(x,y))(u(t,y)-u(t,x)) \dmu(y), &
 x\in \HH,t\geq 0; \\[10pt]
 u(0,x)=u_0(x), & x\in \HH,
\end{cases}
\end{equation}
corresponding to the rescaled kernels 
\begin{equation}
\label{def.JEps}
J_\eps(r):=\eps^{-N-2}J\left(\frac{r}{\eps}\right)
\end{equation}
converge uniformly to the solutions of a local heat-Beltrami equation.

Our purpose is to include another term in \eqref{bandle.NonlocalDiff}, corresponding to a non-local convection effect. In this sense, we notice that, both in the Euclidean and the hyperbolic cases, the diffusion kernel $J(d(x,y))$ is symmetric in $x$ and $y$. On the other hand, the Euclidean non-local convection kernel in \eqref{IntrotranspNonLocRNEps}, namely $G(x,y)=\widetilde{G}(y-x)$, is not symmetric with respect to $x$ and $y$. In this way, it is possible to get a non-zero first moment vector $X_G=-\int_{\RR^N} \widetilde{G}(x) x \hspace{0.1cm}\dd x$. In the Euclidean setting, the fact that the kernel $G$ only depends on the difference vector $y-x$ is essential in proving that the sequence $(u_\eps)_{\eps>0}$ remains uniformly bounded in the $L^2$ norm as $\eps\to 0$. In fact, the following weaker symmetry property for $G$ is the one that plays this crucial role:
\begin{equation}\label{intro.GSymmetricIntegral}
\int_{\RR^N} [G(x,y)-G(y,x)] \dd x=0,\,\forall y \in \RR^N.
\end{equation}
\begin{definition}
We call $G$ a \textit{dissipative kernel} if the integral in \eqref{intro.GSymmetricIntegral} is well defined and is null.
\end{definition}
See Theorem \ref{nonLocalExistence} and Remark \ref{remark:dissipativeKernel} for the explanation of this terminology.

The main difficulty in adapting the model \eqref{IntrotranspNonLocRNEps} to the hyperbolic space, by adding a convection term to \eqref{bandle.NonlocalDiff} while keeping \eqref{intro.GSymmetricIntegral} valid, arises from the lack of an obvious analogue for the vector $y-x$ connecting the points $x$ and $y$. To overcome this issue, we came up with a construction via the geodesic flow on the hyperbolic space, taking into account that the injectivity radius of this space is infinite. More precisely, we make the following assumption on the kernel $G$: 
\begin{hyp}\label{G1.intro}
There exists a function $\widetilde{G}:T\HH \to [0,\infty)$ which is of class $C^1$ outside the zero section (i.e., on $T\HH\setminus\{(x,0): x\in \HH\}$), invariant under the Riemannian geodesic flow $(\Phi_t)_{t\in \RR}$ on $\HH$, such that, for every two distinct points $x,y\in \HH$,
\begin{equation}
\label{G1.intro.eq} G(x,y):=\widetilde{G}(x,V_{x,y}),
\end{equation}
where $V_{x,y}\in T_x\HH$ is the unique vector transporting $x$ to $y$ through the Riemannian exponential mapping (i.e., $\exp_x(V_{x,y})=y$).
\end{hyp}
Throughout the paper, we will assume (if not stated otherwise) that the function $\widetilde{G}:T\HH\to [0,\infty)$ has the kind of $C^1$ regularity specified above.

Detailed definitions and related results are given in Section \ref{section:geodesicFlow}. The hyperbolic analogue of dissipativity condition \eqref{intro.GSymmetricIntegral} is satisfied in this setting, as proven in Proposition \ref{GSymmetricIntegral}.

\subsection*{Linear transport on $\HH$}
In the first part of the paper, we consider the linear transport problem (i.e., we fix $J\equiv  0$ and $f(r)=r$ in \eqref{IntroNonlocalConvDiff}), which is the archetype for a convective non-local problem containing the essential difficulties we will face in the general case. It takes the following form:
\begin{equation}
\label{IntroNonlocalTransp}
\begin{cases}\displaystyle
 \partial_t u(t,x) = \int_{\HH} G(x,y)(u(t,y)-u(t,x)) \dmu(y), & x\in \HH,t\geq 0;\\[5pt]
u(0,x) = u_0(x), & x\in \HH.
\end{cases}
\end{equation}

The well-possedness in $L^2(\HH)$ of this non-local transport problem follows immediately if the right hand side of equation \eqref{IntroNonlocalTransp} defines a bounded operator on $L^2(\HH)$. This holds under very weak integrability conditions on $G$. Moreover, Hypothesis \ref{G1.intro} guarantees the crucial dissipativity property, i.e., that the $L^2(\HH)$ norm of the solution does not increase (see Theorem \ref{nonLocalExistence} for details).

\subsection*{Rescaling the kernel}
In order to rescale the problem \eqref{IntroNonlocalTransp}, we introduce the following kernel, which gives, as $\eps$ tends to $0$, more weight to the movement of particles which are close to each other (and hence having the distance  $d(x,y)=|V_{x,y}|$ smaller):
\begin{equation}
\label{defGEps.intro}
G_\eps(x,y)=\eps ^{-N-1}\widetilde{G}\left(x,\frac{1}{\eps}V_{x,y}\right), \, \eps\in (0,1).
\end{equation}
With this notation, the rescaled problem takes the form:
\begin{equation}
    \label{nonLocalEps.intro}
    \begin{cases}\displaystyle 
\partial_t u^\eps(t,x)=\int_\HH G_\eps (x,y) (u^\eps(t,y)-u^\eps(t,x)) \dmu(y), & x\in \HH,  t\geq 0;\\[10pt]
    u^\eps(0,x)=u_0(x), & x\in \HH.
    \end{cases}
\end{equation}
In order to obtain the convergence of the sequence of non-local solutions $(u_\eps)_{\eps>0}$ towards the local one, we further make some integrability and regularity assumptions concerning $\widetilde G$:

\begin{hyp}\label{G2.intro}
Let us denote $k_{\widetilde{G}}:[0,\infty)\to [0,\infty]$,
\[k_{\widetilde{G}}(r):=\sup_{x\in \HH, |W|=r}\widetilde{G}(x,W)\]
where $|W|$ stands for the hyperbolic norm of the tangent vector $W\in T_x\HH$.
We assume that:
\begin{equation}
\label{G2.intro.eq}
  M(\widetilde{G}):=\Vol(\mathbb{S}^{N-1})\int_0^\infty k_{\widetilde{G}}(r)  (1+r) \left(e^r\sinh(r)\right)^{N-1}  \dd r <\infty,
\end{equation}
where $\mathbb{S}^{N-1}$ stands for the unit sphere in $\RR^N$.
\end{hyp}

We note that, in the context of the hypothesis above, the kernel $G$ does not have to be bounded or $L^1$ on $\HH\times \HH$.

\begin{hyp}\label{G3.intro}
The function $\widetilde{G}$ is such that the \emph{first moment vector field} $X_G$ on $\HH$ defined below is of class $C^1$:
\begin{align}
\label{G3.intro.eq}
X_G(x)=-\int_{T_x\HH} \widetilde{G}(x,W) W \dd W,&&\forall x\in\HH.
\end{align}
\end{hyp}

In this setting, we can formulate the first main result of this paper:

\begin{theorem}
\label{transportConvergence.intro}
Let $G$ satisfy Hypotheses \ref{G1.intro},  \ref{G2.intro} and \ref{G3.intro}. For any
 $u_0\in L^2(\HH)$ and every $T>0$,  the family of solutions $(u^\eps)_{\eps>0}$ of the problem \eqref{nonLocalEps.intro}  converges weakly in $L^2([0,T],L^2(\HH))$, as $\eps\rightarrow 0$, to the unique weak solution (in the sense of Definition \ref{def.weakTranspX}) of the local transport problem:
 \begin{equation}
\label{IntrotranspX}
\begin{cases}
\partial_t u(t,x) = -\dv(u(t)X_G)(x), & x\in \HH, t\geq 0;\\[10pt]
 u(0,x)=u_0(x), & x\in \HH.
\end{cases}
\end{equation}
\end{theorem}

\subsection*{Convection-diffusion processes on $\HH$} In the second part of the paper, we investigate the same concentration phenomenon for the more general family of non-local, non-linear convection-diffusion problems (i.e., $J \neq 0$, $f(r)=|r|^{q-1}r$, $q\geq 1$), where $\eps\in (0,1)$:
\begin{equation}
\label{into.NonlocalConvDiff.eps}
\begin{cases}\displaystyle
\begin{aligned}
 \partial_t u^\eps(t,x) = {}&\eps^{-N-2}\int_\HH J\left(\frac{d(x,y)}{\eps}\right)(u^\eps(t,y)-u^\eps(t,x)) \dmu(y)\\
    & \quad+ \int_{\HH} G_\eps(x,y)(f(u^\eps(t,y))-f(u^\eps (t,x))) \dmu(y)
 , \end{aligned} &  x\in \HH,t\geq 0;\\
u(0,x)= u_0(x), & x\in \HH.
\end{cases}
\end{equation}
Our second main result is the convergence of the solutions $(u_\eps)_{\eps>0}$ of the non-local problem \eqref{into.NonlocalConvDiff.eps} to the ones of a local convection-diffusion problem, under the following hypothesis on $J$:
\begin{hyp}
\label{J1.intro}
The function $J:[0,\infty)\to [0,\infty)$ is continuous, $J(0)>0$, and
\begin{equation}
\label{J1.intro.eq}
\widetilde{M}(J):=\Vol(\mathbb{S}^{N-1})\int_0^\infty J(r)\,(1+ r^2) \left(e^r\sinh(r)\right)^{N-1} \dd r <\infty.
\end{equation}
\end{hyp}

\begin{theorem}
\label{convergenceConvDiff.intro}
Let $J$ satisfy Hypothesis \ref{J1.intro} and $G$ satisfy Hypotheses \ref{G1.intro},  \ref{G2.intro} and \ref{G3.intro}. Let $f(r)=|r|^{q-1} r$, $q\geq 1$ and $u_0\in L^1(\HH)\cap L^\infty(\HH)$.

For every $T>0$, the family $(u^\eps)_{\eps>0}$ of $C^1([0,\infty),L^2(\HH))$ solutions of \eqref{into.NonlocalConvDiff.eps} converges weakly in $L^2([0,T],L^2(\HH))$ and strongly in $L^2([0,T],L^2_\loc(\HH))$ to the unique weak solution (in the sense of Definition \ref{def.weakLocalConvDiff}) of the following local convection-diffusion problem:
\begin{equation}
\label{localConvDiff.intro}
\left\{\begin{aligned}
& u_t(t,x) = A_J \Delta u(t,x) -\dv(f(u(t)) X_G)(x), & x\in \HH, t\geq 0; \\
& u(0,x)=u_0(x), & x\in \HH,
\end{aligned}\right.
\end{equation}
 where $\Delta$ stands for the Laplace-Beltrami operator on $\HH$ (see Section \ref{section:functionSpaces} for the sign convention), the diffusivity constant $A_J$ is:
\[
A_J=\frac{1}{2N}\int_{\RR^N} J(|V|_e)|V|_e^2\,\dd V,\]
and the first moment vector field $X_G$ is given in \eqref{G3.intro.eq}.
\end{theorem}

\begin{remark} In the context of Theorems \ref{transportConvergence.intro} and \ref{convergenceConvDiff.intro}, we can relax the integrability conditions on $J$ and $G$ by requiring $J$ to be bounded from below by a positive constant in a neighbourhood of zero and to satisfy:
\[\int_0^1 J(r) r^{N-1} \dd r <\infty \hspace{0.8cm}\text{and}\hspace{0.8cm} \int_1^\infty J(r) e^{\eps_0 r} \dd r<\infty,\]
while for $G$ we require:
\[\int_0^1 k_{\widetilde{G}}(r) r^{N-1} \dd r <\infty \hspace{0.8cm}\text{and}\hspace{0.8cm} \int_1^\infty k_{\widetilde{G}}(r) e^{\eps_0 r} \dd r<\infty\]
for a constant $\eps_0>0$. The conclusions of the aforementioned theorems hold if we start from $\eps$ small enough instead of $\eps\in (0,1)$.
\end{remark}

\begin{remark}
Under additional constraints for $G$ in terms of $J$, as in \cite{IgnatIgnat2017}, Theorem \ref{convergenceConvDiff.intro} holds for more general non-linear terms, i.e., $f$ can be any non-decreasing locally Lipschitz function. This is done in Section \ref{sec:generalF}.
\end{remark}

We emphasize that the presence of the non-local diffusion term leads to the strong convergence on compact sets for the sequence $(u_\eps)_{\eps>0}$ and also provides more regularity to the limit function. This better behaviour follows from  a compactness result (Theorem \ref{compactnessResultM}) inspired from \cite{IgnatIgnatStancuDumitru} and \cite{rossiNonLocal}, where the authors deal with the Euclidean setting. We have adapted those results to the general case of complete Riemannian manifolds, using chart manipulation techniques as in \cite{SobolevNormsManifolds}. The condition imposed on the sequence of functions in order to obtain compactness resembles those considered by the authors of \cite{MR3586796} and \cite{rossiNonLocal}.

\subsection*{Structure of the paper}
Section \ref{section:theSpace} is a brief introduction to the Riemannian geometry of the hyperbolic space, describing two models for it: the Poincar\' e ball and the upper half-space. We also recall the properties of the geodesic flow on $T\HH$ and describe the invariant functions.

In Sections \ref{section.localLinearTransport}, \ref{section:nonlocalLinearTransport} and \ref{section:relaxationTransport}, we study the local and the non-local transport problems, proving the convergence of the latter to the former. Section \ref{section.exampleOfG} is dedicated to the construction of a rich class of interaction kernels $G$ which satisfy all the hypotheses \ref{G1.intro}, \ref{G2.intro} and \ref{G3.intro}.

Section \ref{section:compactenessManifolds} contains the compactness result.
Sections \ref{section:localNonLinearConvDiff}, \ref{section:nonlocalNonLinearConvDiff} and \ref{section:relaxationConvDiff} contain the analysis of the local and non-local convection-diffusion problems and the second convergence result.

\subsection*{Acknowledgement} The authors would like to thank the anonymous reviewer for providing valuable feedback that, in our opinion, significantly enhances the quality of the paper.

\section{The hyperbolic space $\HH$}
\label{section:theSpace}
\subsection{Function spaces on Riemannian manifolds}
\label{section:functionSpaces}
A Riemannian metric $g$ on a smooth $N$-dimensional manifold $M$ is a family of scalar products in each tangent space $T_xM$, varying smoothly with the base point $x$. In local coordinates $x_1,\ldots, x_N$, this means that for every $1\leq i,j\leq N$, the scalar products of the coordinate vector fields,
$g_{ij}(x)=g\left(\frac{\partial}{\partial x_i}, \frac{\partial}{\partial x_j}\right)$, are smooth functions of $x$ such that $g_{ij}=g_{ji}$ and the matrix $(g_{ij}(x))_{i,j}$ is positive definite.

The Riemannian density on $M$ is defined as follows: $\dd\mu_g(x)(X_1,\ldots,X_N)$ is the volume with respect to the $g$ of the parallelepiped spanned by the vectors $ X_1,\ldots,X_N$ in the vector space $T_x M$. In local coordinates, 
\[\dd\mu_g(x)=\sqrt{\det g_{ij}(x)}\,\dd x_1\ldots \dd x_N.\]
The function spaces $L^p(M)$ for $p\in [0,\infty]$ are defined in an obvious way.

To simplify the notation, we denote $X\cdot Y := g(X,Y)$ the metric product of the vectors $X$ and $Y$. Next, using the metric tensor $g$, we can define the Riemannian norm of a vector $X$ as:
\[|X|_g=\sqrt{X\cdot X}.\]
This leads to the definition of the spaces $L^p(M,TM)$ of vector fields on $M$.
In particular, for $p=2$, $L^2(M)$ and $L^2(M,TM)$ are Hilbert spaces, with the following scalar products:
\begin{align*}
(u_1, u_2)_{L^2(M)}=\int_M u_1(x) u_2(x)\dd\mu_g(x),&&(X_1, X_2)_{L^2(M,TM)}=\int_M X_1(x)\cdot X_2(x)\dd\mu_g(x).
\end{align*}
By abuse of notation, we will also denote by $\|\cdot\|_{L^p(M)}$ the norm of a vector field in $L^p(M,TM)$.
 
The differential of a smooth function $u$ does not depend on the metric, and is given locally by $du=\sum_{i=1}^N \frac{\partial u}{\partial x_i} dx_i$. The Riemannian gradient of $u$, denoted by $\nabla u$, is the vector field dual to $du$ with respect to $g$. More precisely, for every tangent vector $X$,
\[\nabla u\cdot  X = X(u).\]
The expression in local coordinates is the following:
\[\nabla u(x) = \sum_{i,j=1}^N \frac{\partial u}{\partial x_i} g^{ij}(x) \frac{\partial}{\partial x_j},
\]
where $g^{ij}$ are the coefficients of the inverse matrix of $(g_{ij}(x))_{i,j}$.

The Riemannian divergence operator is a first-order differential operator on vector fields defined as the \emph{negative} of the adjoint of the gradient operator with respect to the $L^2$ products.
In coordinates, for a $C^1$ vector field $X=\sum_{l=1}^N X^l(x) \frac{\partial}{\partial x_l}$,
\begin{equation}
\label{riemannianDiv}
\dv_g(X)= \frac{1}{\sqrt{\det g_{ij}}} \sum_{l=1}^N\frac{\partial}{\partial x_l} (X^l\sqrt{\det g_{ij}}).
\end{equation}

Finally, the Laplacian of a function is defined as $\Delta_g u = \dv_g\nabla_g u$. Note that the sign of $\dv_g$ and $\Delta_g$ used here adopts the so-called analyst's convention, making $\Delta_g$ into a non-positive operator in $L^2$.

The weak gradient of a $L^1_\loc (M)$ function $u$ (if it exists) is defined to be the unique $L^1_\loc(M)$ vector field satisfying:
\[\int_M \nabla_g u\cdot X \dd \mu_g(V)=-\int_M u\, \dv_g X \dd\mu_g(V),\]
for every compactly supported smooth vector field $X\in C^\infty_c(M,TM)$.

Next, we give the definition of the Sobolev spaces:
\[W^{1,p}(M)=\left\{ u\in L^p(M) : |\nabla_g u|_g\in L^p(M)\right\},\, p\in [1,\infty],\]
where, here, by $\nabla_g u$ we understand the weak Riemannian gradient of $u$. 
The norm on this Sobolev space is the usual one:
\[\|u\|_{W^{1,p}(M)}=\|u\|_{L^p(M)}+\|\nabla_g u\|_{L^p(M)}.\]

We note that, if the manifold $M$ is complete, the $C_c^\infty(M)$ functions are dense in $W^{1,p}(M)$, for $p\in [1,\infty)$. See, for example, \cite[Satz 2.3]{eichhornSobolev}.

For $p=2$, we have the following characterization, which follows by the Hahn-Banach extension theorem:
\begin{equation}
\label{H1WithDiv}
H^1(M)=\left\{u\in L^2(M): \exists C_u\geq 0,\ ( u, \dv_gX)_{L^2(M)}\leq C_u\|X\|_{L^2(M)},\forall X\in C^\infty_c(M,TM)\right\}.
\end{equation}
Moreover, the $L^2(M)$ norm of the weak gradient of $u$ is the minimum of the admissible values for $C_u$ above.

\subsection{The hyperbolic space. Two isometric models}
We recall some classical aspects about the hyperbolic space $\HH$ and its Riemannian geodesic flow. We begin with a brief presentation of two models of 
$\HH$, each of them to be used when most convenient in specific computations.
Throughout the paper, if not stated otherwise, the operators $\nabla$, $\dv$, $\Delta$ and the norm $|\cdot|$ correspond to the hyperbolic metric, whereas $\nabla_e$, $\dv_e$, $\Delta_e$ and the norm $|\cdot|_e$ are Euclidean. The symbol ``$\cdot$'' stands for the metric scalar product of vectors, either Riemannian or Euclidean, depending on the context.

\subsubsection{The Poincar\' e ball model}

The supporting set for the Poincar\' e ball model of the hyperbolic space $\HH$ is the open unit ball $B^N\subset \RR^N$. At every point $x\in \HH$, the tangent space $T_x\HH$ is canonically identified with $\RR^N$, and the metric tensor is defined by the diagonal matrix
$g_{ij}=\lambda^2(x)\delta_{ij},$
where $\lambda$ is the radial function defined by
\[\lambda(x)=\frac{2}{1-|x|_e^2}.\]

The expressions of the hyperbolic gradient, divergence and Laplacian in this model are as follows:
\begin{align*}
\nabla f = \frac{1}{\lambda^2} \nabla_e f,&&
\dv(Y)=\frac{1}{\lambda^N} \textstyle{\dv_e}(\lambda^N Y),&&
\Delta f = \frac{1}{\lambda^N} \textstyle{\dv_e} \left(\lambda^{N-2} \nabla_e f\right),
\end{align*}
Integration on $\HH$ and on its tangent space at a point $x$ in this model are defined with respect to the volume form $\dmu(x)=\lambda^N(x)\dd x$, respectively $\dmu(V) = \lambda^N(x)\dd V$:
\begin{align*}
\int_\HH f(x)\,\dmu(x)=\int_{B^N} f(x) \lambda^N(x) \,\dd x,&&
&\int_{T_x\HH} f(V)\, \dmu(V)=\lambda^N(x) \int_{\RR^N} f(V) \,\dd V.
\end{align*}

The \emph{boundary at infinity} of $\HH$ is the set of half-infinite geodesics modulo the equivalence relation of being asymptotically close to each other. In the unit ball model, the boundary at infinity is the unit sphere $\partial \HH \simeq \mathbb{S}^{N-1}$. 
A nonconstant (unparametrised) oriented geodesic in $\HH$ is uniquely determined by its initial and final points $\sigma^-\neq\sigma^+\in\partial\HH$.

\subsubsection{The upper half-space model}
The supporting set for this model is 
$\RR^N_+=\{x=(x',x_N)\in \mathbb R^{N-1}\times (0,\infty)\}$,
with Riemannian metric defined by
\[g_{ij}(x)=\frac{1}{x_N^2}\delta_{ij}.\]
In this setting, the expressions of the gradient, divergence and Laplacian are:
\begin{align*}
\nabla f = x_N^2 \nabla_e f,&& \dv(Y)=x_N^N \,\textstyle{\dv_e}\left(\frac{1}{x_N^N} Y\right), && \Delta f = x_N^N\, \textstyle{\dv_e} \left(\frac{1}{x_N^{N-2}} \nabla_e f\right).
\end{align*}
The volume form on $\HH$ and on its tangent space at a point $x$ become $\dmu(x)= x_N^{-N}\, \dd x$, respectively $\dmu(V)=x_N^{-N}\, \dd V$.
In this model, the boundary at infinity is the one-point compactification of the hyperplane $\{x_N=0\}$, that is $\partial \HH \simeq \overline{\RR^{N-1}}$.

These two models of $\HH$ are isometric; an example of isometry between them is the Cayley transform $\mathcal{C}:\mathbb R^N_+ \to B^N$, \begin{equation}
\label{isometryBetweenModels}
\mathcal{C}(x',x_N)=\left(\frac{2x'}{1+|x|_e^2+2x_N},\frac{|x|_e^2-1}{1+|x|_e^2+2x_N}\right),\quad \text{for } x=(x',x_N)\in \mathbb R^N_+.
\end{equation}
The isometry extends to a diffeomorphism between the boundaries at infinity by setting $y_N=0$, amounting to the inverse of the stereographic projection. It is a conformal diffeomorphism, reflecting the fact that $\partial\HH$ has a well-defined conformal class, but not a preferred metric.

\subsection{The geodesic flow on $\HH$}
\label{section:geodesicFlow}

We recall some facts about the Riemannian geodesic flow on $\HH$ and we give a characterization of the functions which are invariant under the flow. 

\begin{definition}
\label{defGeodesicFlow}
For every $(x,V)\in T\HH$, let $\gamma_{x,V}$ be the unique geodesic such that $\gamma(0)=x$ and $\gamma'(0)=V$. Moreover, for $V\neq 0$ let $\sigma^-(x,V)$, $\sigma^+(x,V)\in \partial \HH$ be the initial and final points at infinity of the geodesic $\gamma_{x,V}$ (refer to Figure \ref{fingure.poincareDisk}).
The geodesic flow $\Phi_t(x,V)$ emerging from the point $(x,V)$ in the tangent bundle, at time $t\in \RR$, has the following form:
\[\Phi_t(x,V)=(\gamma_{x,V}(t),\gamma_{x,V}'(t)).\]
We note that $\gamma_{x,V}(t)=\exp_x(tV)$.
\end{definition}

\begin{definition}
\label{invariantToGF}
We call a function $\widetilde{G}:T\HH\to \RR$ \emph{invariant with respect to the geodesic flow} $(\Phi_t)_{t\in \RR}$ on $\HH$ if, for every $t\in \RR$ and $(x,V)\in T\HH$,
\[\widetilde{G}(x,V)=\widetilde{G}(\Phi_t(x,V)).\]
\end{definition}

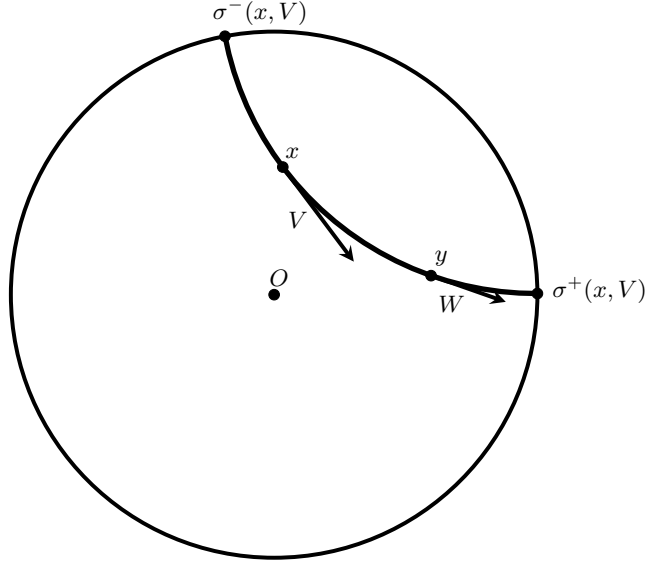
\begin{figure}[!ht]
\centering
\color{black}
\definecolor{xdxdff}{rgb}{0,0,0}
\definecolor{xdxdff}{rgb}{0,0,0}
\definecolor{uuuuuu}{rgb}{0,0,0}
\begin{tikzpicture}[line cap=round,line join=round,>=triangle 45,x=3.5cm,y=3.5cm]
\clip(-1.1247000368373487,-1.1353696047142738) rectangle (1.473424059207674,1.2103440906620705);
\draw [line width=1.5pt] (0,0) circle (3.5cm);
\draw [shift={(0.994705722146539,1.2070279123112277)},line width=2pt]  plot[domain=3.329472115499865:4.716783160626449,variable=\t]({1*1.2026453570232034*cos(\t r)+0*1.2026453570232034*sin(\t r)},{0*1.2026453570232034*cos(\t r)+1*1.2026453570232034*sin(\t r)});
\draw [-stealth,line width=1.5pt] (0.032790591869119035,0.4851725688296943) -- (0.3026598977390833,0.12555568594874844);
\draw [-stealth,line width=1.5pt] (0.5960946621733908,0.07236272925652965) -- (0.879925349285421,-0.02734778687832512);
\begin{scriptsize}
\draw [fill=uuuuuu] (0,0) circle (2pt);
\draw[color=uuuuuu] (0.019373888889911765,0.06497833469533648) node {$O$};
\draw [fill=xdxdff] (-0.1867760945330774,0.9824025094180954) circle (2pt);
\draw[color=xdxdff] (-0.05245941000704148,1.0741329514753544) node {$\sigma^-(x,V)$};
\draw [fill=xdxdff] (0.999990345605536,0.004394166100701227) circle (2pt);
\draw[color=xdxdff] (1.237130233300994,0.017246202907030037) node {$\sigma^+(x,V)$};
\draw [fill=xdxdff] (0.032790591869119035,0.4851725688296943) circle (2pt);
\draw[color=xdxdff] (0.06960305125203092,0.5372818360212774) node {$x$};
\draw[color=black] (0.09119531008211604,0.2849902124000856) node {$V$};
\draw [fill=xdxdff] (0.5960946621733908,0.07236272925652965) circle (2pt);
\draw[color=xdxdff] (0.6353202326002608,0.13702891824965997) node {$y$};
\draw[color=black] (0.674186298494414,-0.034346055923055085) node {$W$};
\end{scriptsize}
\end{tikzpicture}
\caption{\label{fingure.poincareDisk}A geodesic through the point $x$ in the Poincar\' e ball model, tangent to $V$, together with its initial and final points. The geodesic flow transports $(x,V)\in T\HH$ to $(y,W)\in T\HH$, i.e., $\Phi_t(x,V)=(y,W)$ for some $t>0$.}
\end{figure}

For every $0\neq V\in T_x\HH$, the curve $(\Phi_t(x,V))_{t\in\RR}$ describes a parametrised geodesic curve of constant speed $|V|$ which originates at the point at infinity $\sigma^-(x,V)\in\partial\HH$, passes through $x$ at $t=0$, and ends at $\sigma^+(x,V)\in \partial\HH$ (This actually defines a diffeomorphism $\Phi_x(\cdot)=\sigma^+(x,\cdot)$ from the unit sphere $T^1_x\HH$ to $\partial\HH$. The composition $\Phi_y^{-1}\circ \Phi_x$ turns out to be a conformal diffeomorphism, thus defining a conformal structure on $\partial\HH$).
Therefore, we are able to fully characterize the functions defined on $T\HH$ which are invariant with respect to $(\Phi_t)_{t\in \RR}$ by the following:
\begin{proposition}
\label{GInvariantGFq}
A function $\widetilde{G}:T\HH \to \RR$ is invariant with respect to the geodesic flow $(\Phi_t)_{t\in \RR}$ if and only if there exists a function $g:\partial \HH \times \partial \HH \times (0,\infty)\to \RR$ such that for all $V\neq 0$ in $T_x\HH$,
\[\widetilde{G}(x,V)=g(\sigma^-(x,V),\sigma^+(x,V),|V|).\]
\end{proposition}
\begin{proof}
Refer again to Figure \ref{fingure.poincareDisk}. There is a one-to-one correspondence between the functions which are invariant to the geodesic flow and those defined on the set of trajectories 
\[\left\{(\Phi_t(x,V))_{t\in\RR}:(x,V)\in T\HH\right\}.\]
The projection of such a trajectory on the manifold is a geodesic and the flow preserves the hyperbolic length of vectors. Every non-constant trajectory is uniquely characterised by a pair of distinct points at in infinity $\sigma^-$, $\sigma^+$, together with the energy level, i.e., the speed $|V|$.
\end{proof}

\subsection{Properties of the exponential mapping}
The aim of the following lemma is to compute the differential of the exponential mapping in $\HH$.
Let $P(x,y)$ denote the parallel transport in $T\HH$ along the unique geodesic between the points $x$ and $y$ in $\HH$.

\begin{lemma}
\label{diffExp}
For every $x\in \HH$ and $0\neq W\in T_x\HH$, the differential of exponential mapping acts as follows:
\[\dd_W(\exp_x)(V)=\left\{\begin{aligned}
&P(x,\exp_x(W))(V), & V=\alpha W \in {\rm span}\{W\}\\
&\frac{\sinh(|W|)}{|W|}(P(x, \exp_x(W))(V), & V\bot W.
\end{aligned}\right.\]
In particular, the determinant of the Jacobian matrix of $\exp_x$ at $W$ is:
\[J_{\exp_x}(W)=\left(\frac{\sinh(|W|)}{|W|}\right)^{N-1}.\]
\end{lemma}
\begin{proof}
We work in the Poincar\' e ball model and we first apply a hyperbolic isometry such that $x$ becomes the origin of the Poincar\'e ball. 
The rays originating in $x=0$ are geodesics, so
\begin{align*}
\exp_0(W)=\frac{\tanh\left(|W|_e\right)}{|W|_e}W. 
\end{align*}
We can now compute easily for $V\perp W$:
\begin{align*}
P(0,\exp_0(W))(V)=\frac{1}{\cosh^2\left(|W|_e\right)}V,
\end{align*}
The conclusion follows by direct computation.
\end{proof}
Of course, this lemma is independent of the model we choose for the hyperbolic space.\\

The next lemma provides some bounds on the integrals which contain the exponential mapping and will be used during the proofs of Theorems \ref{transportConvergence.intro} and \ref{convergenceConvDiff.intro}.

\begin{lemma}
\label{lemmaJacobianExpFixedV}
Consider a fixed vector $V\in \RR^N$. Identifying the tangent space in each point of the half-space model with $\RR^N$, we can view $V$ as a vector field on the hyperbolic space. Then, for every non-negative measurable function $\Psi:\HH\to [0,\infty)$,
\[\int_\HH \Psi(\exp_y(y_N V)) \dmu(y) \leq e^{(N-1)|V|_e} \int_\HH \Psi(z) \dmu(z).\]
\end{lemma}
\begin{proof}
At every point $y=(y',y_N)\in \RR^N_+$, the vector $y_N V\in T_y\HH\simeq \RR^N$ has length $|y_N V|=|V|_e.$
We denote by $\mathcal{T}_V:\HH\to \HH$ the hyperbolic exponential of $y_N V$:
\[\mathcal{T}_V(y)=\exp_y(y_N V).\]
Writing $V=(V',V_N)$ and $y=(y',y_N)$, it is straightforward to compute
\[\mathcal{T}_V(y',y_N)=\left(y'+y_N\frac{\sinh|V|_e  C_V}{|V|_e} V', y_N C_V\right),\]
where
\[C_V:=\frac{\cosh|V|_e + \sin\theta\cdot  \sinh|V|_e}{\cos^2\theta\cdot \cosh^2|V|_e+\sin^2(\theta)}.\]
 Here, $\theta$ stands for the angle between $V$ and the horizontal hyperplane. It follows that the determinant of the Jacobian of $\mathcal{T}_V$ at $y$ equals $C_V$, which, in particular,  is independent of $y$. By the change of variables $z=\mathcal{T}_V(y)$, we obtain that $\dd z=C_V \dd y$ and $z_N=C_Vy_N$, hence $\dmu(z)=C_V^{1-N}\dmu(y)$ and the conclusion follows from the evident estimate for the constant $C_V$ defined above:
 \[e^{-|V|_e}\leq C_V\leq e^{|V|_e}.\qedhere\]
\end{proof}

\section{Local linear transport on the hyperbolic space}
\label{section.localLinearTransport}
Let $X$ be a vector field on $\HH$.
We consider the following local transport Cauchy problem in divergence form:
\begin{equation}
\label{transpX}
\begin{cases}
\partial_t u(t,x) = -\dv(u(t) X)(x) , & x\in \HH, t\geq 0;\\
u(0,x)=u_0(x), & x\in \HH.
\end{cases}
\end{equation}
For the Euclidean case we refer to \cite[Ch. 3]{bahouri} where, besides the classical theory of Lipschitz vector fields, the authors consider an extension to less regular vector fields i.e. log-Lipschitz. For clarity we consider here the case of $C^1$-vector fields, even though the results can be easily extended to the Lipschitz case.

\subsection{Existence and uniqueness for the local problem}
First, we recall a standard result that guarantees the existence of classical solutions for the problem \eqref{transpX}, in the more regular case of a bounded $C^2$ vector field $X$ and for initial data $u_0\in C^1(\HH)$.
\begin{proposition}
\label{transpXExistence}
Let $X$ be a bounded $C^2$ vector field on $\HH$. 
For $u_0\in C^1(\HH)$, the problem \eqref{transpX} admits a classical solution in $C^1(\RR\times \HH)$.
\end{proposition}
\begin{proof}
Since $X$ is a bounded and locally Lipschitz vector field on the complete manifold $\HH$, its flow, denoted by $(\Phi_t^X)_{t\in \RR}$, is defined for all $t\in\RR$.
If we denote $\alpha(x):=\dv(X)(x)$, the equation becomes:
\[\partial_t u(t,x)=-\alpha(x) u(t,x)-\nabla u(t,x) \cdot X(x).\]
Hence in terms of the vector field $Y:=\partial_t+X$ on $\RR\times \HH$, the equation can be written as:
\[Y(u)(t,x)=-\alpha(x)u(t,x).\]
Now, $Y$ is also bounded and $C^2$ on the complete Riemannian manifold $\RR\times \HH$, therefore the above equation has a unique solution given in terms of the flow $(\Phi^Y_t)_{t\in \RR}$, starting from the non-characteristic hypersurface $\{0\}\times\HH$, by:
\[u\left(\Phi_t^Y(0,x)\right)=\exp\left({-\int_0^t \alpha\left(\Phi_\tau^Y(0,x) \right) \dd \tau }\right) u(0,x).\]
Since $\Phi_t^Y(0,x)=\left(t,\Phi_t^X(x)\right)$, it follows that:
\begin{equation}
\label{transpXExplicitSol}
u(t,x)=\exp\left({-\int_0^t \alpha\left(\Phi_{-\tau}^X(x) \right) \dd \tau }\right) u_0(\Phi_{-t}^X(x)),
\end{equation}
which is indeed a $C^1$ solution of \eqref{transpX}.
\end{proof}

Next, we introduce the definition of weak solutions for the problem \eqref{transpX} and we prove the existence of such solutions for $L^2_\loc$ initial data. The concept of weak solution for the transport problem \eqref{transpX} will be useful for identifying the limit in the convergence result, i.e., in the proof of Theorem \ref{transportConvergence.intro}.
\begin{definition}
\label{def.weakTranspX}
Let $u_0\in L^2_\loc (\HH)$ and $X$ a bounded $C^1$ vector field on $\HH$. We call $u\in L^2_\loc ([0,\infty)\times\HH)$ a weak solution of \eqref{transpX} if, for every $\varphi\in C_c^1([0,\infty)\times \HH)$,
\begin{equation}
\label{weakTranspX}
\int_0^\infty \int_\HH u(t,x)\left[\partial_t\varphi(t,x)+X(\varphi(t))(x)\right]\dmu(x) \dd t=-\int_\HH u_0(x) \varphi(0,x)\dmu(x).
\end{equation}
\end{definition}
\begin{remark}
An integration by parts argument implies that, if $X\in C^2(\HH,T\HH)$, then classical solutions of \eqref{transpX} are also weak solutions.
\end{remark}
\begin{proposition}
\label{existenceWeakSol} Let $X$ be a bounded $C^1$ vector field.
If $u_0\in L^2_\loc (\HH)$, then the function $u$ given by \eqref{transpXExplicitSol} is in $L^2_\loc ([0,\infty)\times\HH)$ and is a weak solution of \eqref{transpX}.
\end{proposition}
\begin{proof} First, we prove that $u\in L^2_\loc ([0,T)\times\HH)$.
Let $K$ be a compact set in $\HH$ and $T>0$. Taking the $L^2$ norm in \eqref{transpXExplicitSol}, we obtain:
\[\|u(t)\|_{L^2(K)}^2 =\int_K e^{-2\int_0^t \alpha\left(\Phi_{-\tau}^X(x) \right) \dd \tau } |u_0(\Phi_{-t}^X(x))|^2\dmu(x).\]
Changing variables $x=\Phi_{t}^X(y)$, we get that:
\[\|u(t)\|_{L^2(K)}^2 = \int_{\Phi_{t}^X(K)} e^{-2\int_0^t \alpha\left(\Phi_{t-\tau}^X(y) \right) \dd \tau } |u_0(y)|^2 J_{\Phi_{t}^X}(y)\dmu(y).
\]
Next, by Liouville's formula (\cite[Proposition 2.18, p.~152]{chicone}), the determinant of the Jacobian of $\Phi_t^X$ satisfies the following ODE:
\[\frac{d}{\dd t}J_{\Phi_{t}^X}(y)=\dv(X)({\Phi_{t}^X(y)})J_{\Phi_{t}^X}(y).\]
Therefore, we obtain  $J_{\Phi_{t}^X}(y)=e^{\int_0^t \alpha\left(\Phi_{\tau}^X(y) \right) \dd \tau }$ and, thus,
\[\|u(t)\|_{L^2(K)}^2 = \int_{\Phi_{t}^X(K)} e^{-\int_0^t \alpha\left(\Phi_{t-\tau}^X(y) \right) \dd \tau } |u_0(y)|^2 \dmu(y),
\]
which is bounded for $t\in [0,T]$, since the vector field $X$ is $C^1$ and bounded.
Finally, by the change of variables $x=\Phi_{t}^X(y)$ and using the above computation for its Jacobian, we obtain that $u$ satisfies \eqref{weakTranspX}. 
\end{proof}

We now prove a uniqueness result concerning weak solutions, similar to the Euclidean case in \cite[Section 2.1]{evans}.
\begin{theorem}
\label{transpXUniqueness}
Let $X$ be a $C^1$ vector field. For any $u_0$ in  $L^2_\loc (\HH)$, there exists at most one weak solution of \eqref{transpX}, in the sense of Definition \ref{def.weakTranspX}.
\end{theorem}
\begin{proof}
It is enough to prove that, if $u\in L^2_\loc ([0,\infty)\times\HH)$ is such that for every $T>0$ and every $\varphi\in C_c^1([0,T)\times \HH)$,
\begin{equation} \label{uperpfi}
\int_0^T \int_\HH u(t,x) \left[\partial_t\varphi(t,x)+X(\varphi(t))(x)\right]\dmu(x) \dd t=0,\end{equation}
then $u\equiv 0$. We achieve this by considering an arbitrary $T>0$ and $f\in C_c^{\infty}((0,T)\times \HH)$, and solving the following final value problem:
\begin{equation}
\label{eq.varphi}
\left\{\begin{aligned}
&\partial_t \varphi(t,x)+X(\varphi(t))(x)=f(t,x)  ,& x\in \HH, t\in [0,T];\\
&\varphi(T,x)=0  ,& x\in \HH.
\end{aligned}\right.
\end{equation}
The solution of \eqref{eq.varphi} can be constructed explicitly by the method of characteristic curves:
\[\varphi(t,x)=-\int_t^T f(\tau,\Phi_{\tau-t}^X(x)) \dd \tau.\]
Since $X$ is $C^1$ and bounded, it follows that $\varphi\in C_c^1([0,T)\times \HH)$, so we can apply \eqref{uperpfi}:
\[0=(u, \partial_t \varphi+X(\varphi))_{L^2((0,T)\times \HH)} = (u,f)_{L^2((0,T)\times \HH)}.
\]

The conclusion follows since $C_c^{\infty}((0,\infty)\times \HH)$ is dense in $L^2((0,\infty)\times \HH)$.
\end{proof}

\section{Non-local linear transport on the hyperbolic space}
\label{section:nonlocalLinearTransport}
This section is dedicated to the study of the non-local transport problem \eqref{IntroNonlocalTransp}. More exactly, we prove that the problem is well-posed, for a more relaxed assumption on the kernel $G$ and then we return to the setting of Hypotheses \ref{G1.intro} - \ref{G3.intro} to obtain $L^2$-norm decay for the solutions.

\begin{theorem}
\label{nonLocalExistence}
Let $G:\HH\times\HH\to [0,\infty)$ be a positive measurable function satisfying 
\begin{equation}
\label{condition.integrabilityG}
\sup_{x\in \HH}\left[ \int_\HH G(x,y) \dmu(y) + \int_\HH G(y,x) \dmu(y)\right]<\infty.
\end{equation}
Then, for every $u_0\in L^2(\HH)$, there exists a unique solution $u\in C^\infty([0,\infty),L^2(\HH))$ of the non-local transport problem \eqref{IntroNonlocalTransp}.
Moreover, if we assume further that $G$ is a dissipative kernel, namely:
\begin{equation}\label{GSymmetricintegral.eq}
\int_{\HH} [G(x,y)-G(y,x)]\dmu(x)=0,
\end{equation}
the norm $\|u(t)\|_{L^2(\HH)}$ does not increase in time.
\end{theorem}
\begin{remark}
    \label{remark:dissipativeKernel}
    The second statement of this theorem is the reason for which we called $G$ a \textit{dissipative kernel}: the associated evolution problem \eqref{IntroNonlocalTransp} dissipates the $L^2$ energy.
\end{remark}
\begin{proof}
By assumption \eqref{condition.integrabilityG}, there exists some $M>0$ such that, for every $x\in \HH$,
\begin{equation}\label{eq:intGBoundedM}
\int_\HH G(x,y) \dmu(y)\leq M\text{ and }\int_\HH G(y,x) \dmu(y)\leq M.
\end{equation}
It is enough to prove that the operator $L_G$ defined by:
\[L_G(v)(x)\coloneqq\int_{\HH}G(x,y)(v(y)-v(x)) \dmu(y)\]
is a bounded operator on $L^2(\HH)$ and then the problem \eqref{IntroNonlocalTransp} admits a unique solution $u=e^{tL_G}u_0$.

Indeed, for $v\in L^2(\HH)$, we can write:
\[
L_G(v)(x) = L_G^\#(v)(x)- v(x) \int_{\HH} G(x,y) \dmu(y),
\]
where the first term above is defined as follows:
\[
L_G^\#(v)(x) \coloneqq \int_{\HH} G(x,y) v(y) \dmu(y).
\]
By \eqref{eq:intGBoundedM}, we only need to show that the operator $L_G^\#$ is bounded on $L^2(\HH)$. In this sense, we use H\" older's inequality, together with Tonelli's theorem (recall that $G$ is a non-negative function) and obtain:
\begin{align*}
   \|L^\#_G(v)\|^2_{L^2(\HH)} &\leq\int_{\HH}\left[\int_{\HH} G(x,y) \dmu(y) \int_{\HH} G(x,y)v(y)^2 \dmu(y)\right]\dmu(x)\\
   &\leq M \int_{\HH} \int_{\HH} G(x,y)v(y)^2 \dmu(y)\dmu(x)\\
   &\leq M^2 \int_{\HH} v(y)^2 \dmu(y).
\end{align*}
Therefore, the operators $L_G^\#$ and (as a consequence) $L_G$ are well-defined and bounded on $L^2(\HH)$.

For the second statement, we multiply the equation \eqref{IntroNonlocalTransp} by $u$ and integrate:
\[\frac{d}{dt} \left(\frac 12\int_\HH |u(t,x)|^2\dmu(x)\right)=\int_{\HH\times\HH} G(x,y)(u(t,y)-u(t,x))u(t,x) \dmu(y) \dmu(x).\]
We claim that the integral in the right-hand side above equals:
\[\int_{\HH\times\HH} G(x,y)(u(t,x)-u(t,y))u(t,y) \dmu(y) \dmu(x).\]
Indeed, it is enough to notice that, using Tonelli's theorem and the equality \eqref{GSymmetricintegral.eq}, we arrive at:
\[\int_{\HH\times\HH} G(x,y)u(t,x)^2 \dmu(y) \dmu(x) = \int_{\HH\times\HH} G(x,y)u(t,y)^2 \dmu(y) \dmu(x),
\]
which implies the claim above. As a consequence of that statement, we obtain the desired decay estimate:
\[\frac{d}{dt} \|u(t)\|^2_{L^2(\HH)}=-\int_{\HH\times\HH} G(x,y)(u(t,x)-u(t,y))^2 \dmu(y) \dmu(x)\leq 0.\qedhere\]
\end{proof}

Now we prove that that the conditions we imposed for $G$ in the Introduction are sufficient for \eqref{GSymmetricintegral.eq} to hold. More precisely, we have the following:

\begin{proposition}
\label{GSymmetricIntegral}
If $G$ satisfies Hypotheses \ref{G1.intro} and \ref{G2.intro}, then, for almost every $y\in \HH$, \eqref{GSymmetricintegral.eq} is satisfied.
\end{proposition}
\begin{proof}
First, we write the integral above in terms of $\widetilde{G}$:
\[I(y):=\int_{\HH} [G(x,y)-G(y,x)]\dmu(x)=\int_\HH \left[ \widetilde{G}(x,V_{x,y})-\widetilde{G}(y,V_{y,x})\right]\dmu(x).\]
Next, the invariance of $\widetilde{G}$ with respect to the geodesic flow (see Definitions \ref{defGeodesicFlow} and \ref{invariantToGF}) implies that:
\[\widetilde{G}(x,V_{x,y})=\widetilde{G}(\exp_x(V_{x,y}),\gamma_{x,V_{x,y}}'(1)).\]
The definition of $\gamma_{x,V_{x,y}}$ and the uniqueness of geodesics implies that $\gamma_{x,V_{x,y}}'(1)=-V_{y,x}$, so:
\[I(y)=\int_\HH \left[ \widetilde{G}(y,-V_{y,x})-\widetilde{G}(y,V_{y,x})\right]\dmu(x).\]
Using that the Jacobian determinant of the exponential mapping is symmetric (for its exact form, see Lemma \ref{diffExp}), we obtain:
\[I(y)=\int_{T_y\HH} \left[ \widetilde{G}(y,-V)-\widetilde{G}(y,V)\right]\left|J_{\exp_y}(V)\right| \dmu(V),\]
which vanishes using the change of variables $V\to -V$.
\end{proof}
\begin{remark}
A simple calculation implies that, if $G$ satisfies Hypotheses \ref{G1.intro} and \ref{G2.intro}, then:
\[\sup_{x\in \HH}\int_\HH G(x,y) \dmu(y)\leq M(\widetilde{G})\hspace{0.2cm}\text{ and }\sup_{x\in \HH}\int_\HH G(y,x) \dmu(y)\leq M(\widetilde{G}).\]
Therefore, the kernel $G$ also satisfies \eqref{condition.integrabilityG}, so Theorem \ref{nonLocalExistence} is true in the particular setting of $G$ satisfying all the hypotheses \ref{G1.intro}, \ref{G2.intro} and \ref{G3.intro} in the Introduction.
\end{remark}

\section{Relaxation limit for the transport problem}
\label{section:relaxationTransport}
In this section we focus on the family of rescaled problems \eqref{nonLocalEps.intro}. These rescaled problems are all instances of \eqref{IntroNonlocalTransp}, where the kernel
 $G$ is replaced by $G_\eps (x,y)=\eps^{-N-1}\widetilde{G}(x,\frac{1}{\eps}V_{x,y})$. 
If we define the rescaled operator $L_{G_\eps}$ on $L^2(\HH)$ by
\[L_{G_\eps}(\psi)(x)=\eps^{-N-1}\int_\HH \widetilde{G}\left(x,\frac{1}{\eps} V_{x,y}\right) (\psi(y)-\psi(x)) \dmu(y),
\]
then \eqref{nonLocalEps.intro} can be written as:
\begin{equation}
    \label{nonLocalEps}
    \begin{cases}
\partial_t u^\eps(t,x)=L_{G_\eps}(u^\eps(t))(x), & x\in \HH, t\geq 0 \\[5pt]
 u^\eps(0,x)=u_0(x), & x\in \HH.
    \end{cases}
\end{equation}
Notice that, if $\widetilde{G}$ is invariant with respect to the geodesic flow $(\Phi_t)_{t\in \RR}$, then so is the rescaled function:
\[
\widetilde{G}_\eps:T\HH\to \RR,\quad\widetilde{G}_\eps(x,V)=\eps^{-N-1}\widetilde{G}\left(x,\frac{1}{\eps}V\right),
\]
since $\gamma_{x,V}(t)=\eps\gamma_{x,\frac{1}{\eps}V}(t \eps)$. Therefore, we can apply Theorem \ref{nonLocalExistence} to obtain existence, uniqueness and $L^2$ norm decay of solutions of \eqref{nonLocalEps}. We are interested in the behavior of these solutions as $\eps\to 0$.

Before proving Theorem \ref{transportConvergence.intro}, we make the following remark concerning the limit local problem:
\begin{remark}
\label{XLeqMG}
Using Lemma \ref{diffExp}, we get that Hypothesis \ref{G2.intro} guarantees the integrability of the mapping $W\to \widetilde{G}(x,W) W$, for every $x\in \HH$. Therefore, the vector field $X_G$ defined in \eqref{G3.intro.eq} satisfies
\[\|X_G\|_{L^\infty(\HH)}\leq M(\widetilde{G}).\]
This, together with the regularity Hypothesis \ref{G3.intro}, implies that the existence and uniqueness results in Section \ref{section.localLinearTransport} can be applied for the limit problem \eqref{IntrotranspX}.
\end{remark}

The proof of Theorem \ref{transportConvergence.intro} requires the following two lemmas, regarding the adjoint of the operator $L_{G_\eps}$. In view of Proposition \ref{GSymmetricIntegral} and using the Fubini-Tonelli theorem\footnote{By the same reasoning as in the proof of Theorem \ref{nonLocalExistence},
\[\int_{\HH\times\HH} G(x,y) |\psi(x)\psi(y)|\dmu(x)\dmu(y)\leq \frac{1}{2}\int_{\HH\times\HH} G(x,y) (|\psi(x)|^2+|\psi(y)|^2)\dmu(x)\dmu(y)<\infty.\]}, this adjoint has the following expression:
\[L^*_{G_\eps}(\psi)(y)=\int_\HH G_\eps(x,y)(\psi(x)-\psi(y))\dmu(x),
\,\forall \psi\in L^2(\HH).\]

\begin{lemma}
\label{computations.LGEpsAdj}
Let $G$ satisfy Hypotheses \ref{G1.intro} and \ref{G2.intro}. For every function $\psi\in C_c^1(\HH)$,
\[
L^*_{G_\eps}(\psi)(y)  = -\int_{T_y\HH} \widetilde{G}\left(y, W\right)
W\cdot \int_0^1\nabla F_y(-\tau\eps W)\dd \tau \rho(\eps |W|) \dmu(W),
\]
where $F_y=\psi\circ \exp_y$ is the geodesic normal coordinates expression of $\psi$ around $y$ and 
\[\rho(r)=\left(\frac{\sinh(r)}{r}\right)^{N-1}.\]
\end{lemma}
\begin{proof}
By the definition of $G_\eps$,
\[L^*_{G_\eps}(\psi)(y)=\eps^{-N-1}\int_\HH \widetilde{G}\left(x,\frac{1}{\eps} V_{x,y}\right)(\psi(x)-\psi(y)) \dmu(x).\]
Using that $\widetilde{G}$ is constant along the orbits of the geodesic flow, we also have:
\[L^*_{G_\eps}(\psi)(y)=\eps^{-N-1}\int_\HH \widetilde{G}\left(y,-\frac{1}{\eps} V_{y,x}\right)(\psi(x)-\psi(y)) \dmu(x).\]
The change of variables $W=V_{y,x}$, that is $x=\exp_y(W)$, turns the above integral into:
\begin{equation}
\label{LGEpsAdjWithJacobian}
L^*_{G_\eps}(\psi)(y)=\eps^{-N-1}\int_{T_y\HH} \widetilde{G}\left(y,-\frac{1}{\eps} W\right)(\psi(\exp_y(W))-\psi(y)) |J_{\exp_y}(W)| \dmu(W).
\end{equation}
In  view of Lemma \ref{diffExp}, $J_{\exp_y}(W)=\rho(|W|)$.
Using this fact in \eqref{LGEpsAdjWithJacobian}, 
a change of variables $W\rightarrow -\eps W$ and the Fundamental Theorem of Calculus, we obtain:
\begin{equation}
\label{LGEpsStarV1}\begin{aligned}
L^*_{G_\eps}(\psi)(y) & =\eps^{-1}\int_{T_y\HH} \widetilde{G}\left(y, W\right)(F_y(-\eps W))-F_y(0))\rho(\eps |W|) \dmu(W) \\
& = -\int_{T_y\HH} \widetilde{G}\left(y, W\right)
W\cdot \int_0^1\nabla F_y(-\tau\eps W)\dd \tau \rho(\eps |W|) \dmu(W),
\end{aligned}
\end{equation}
which finishes the proof.
\end{proof}

\begin{lemma}
\label{H1L2Boundedness} Let $G$ satisfy Hypotheses \ref{G1.intro} and \ref{G2.intro}. The following holds uniformly in $\eps\in(0,1)$:
\[\|L_{G_\eps}^*(\psi)\|_{L^2(\HH)}\leq M(\widetilde{G}) \|\nabla \psi\|_{L^2(\HH)},\ \forall\psi\in H^1(\HH). \]
%where the constant $M(\widetilde{G})$ introduced in Hypothesis %\ref{G2.intro} is independent of $\eps$.
This means that $L_{G_\eps}^*:H^1(\HH)\to L^2(\HH)$ is a bounded operator with norm at most $M(\widetilde{G})$.
\end{lemma}

\begin{proof}[Proof of Lemma \ref{H1L2Boundedness}]
By density and the boundedness of $L_{G_\eps}^*$ on $L^2(\HH)$, it is enough to prove the lemma for $\psi\in C_c^1(\HH)$.
By Lemma \ref{diffExp}, we obtain that:
\[
\begin{aligned}
\nabla F_y(-\tau\eps W)\cdot W &=\nabla\psi(\exp_y(-\tau\eps W)) \cdot \dd_{-\tau\eps W}(\exp_y)(W)\\
&=\nabla\psi(\exp_y(-\tau\eps W)) \cdot P(y,\exp_y(-\tau\eps W))(W),
\end{aligned}\]
where $P(y,x)$ is the parallel transport in $T\HH$ along the unique geodesic from $y$ to $x$ in $\HH$.
Plugging into Lemma \ref{computations.LGEpsAdj}, we obtain:
\[L^*_{G_\eps}(\psi)(y)=-\int_0^1\int_{T_y\HH} \widetilde{G}\left(y, W\right)P(y,\exp_y(-\tau\eps W))(W)\cdot\nabla\psi(\exp_y(-\tau\eps W))\rho(\eps |W|) \dmu(W) \dd \tau.\]
Since the parallel transport $P$ is an isometry,
\[\left|P(y,\exp_y(-\tau\eps W))(W)\cdot\nabla\psi(\exp_y(-\tau\eps W))\right|\leq |W| |\nabla \psi(\exp_y(-\tau\eps W))|.\]
Taking into account that $\widetilde{G}(x,W)\leq k_{\widetilde{G}}(|W|)$, we obtain:
\[|L^*_{G_\eps}(\psi)(y)|\leq\int_0^1\int_{T_y\HH} k_{\widetilde{G}}(|W|) |W||\nabla\psi(\exp_y(-\tau\eps W))|\rho(\eps |W|) \dmu(W) \dd \tau.\]
Next, we work on the half-space model and we transform the integral on $T_y\HH$ through the identification $T_y\HH\simeq \RR^N$, switching to Euclidean norms via the change of variables $V=\frac{1}{y_N} W$, where we notice that $|W|=|V|_e$. Therefore,
\[|L^*_{G_\eps}(\psi)(y)|\leq\int_0^1\int_{\RR^N} k_{\widetilde{G}}(|V|_e)|V|_e |\nabla\psi(\exp_y(-\tau\eps y_N V))|\rho(\eps |V|_e) \dd V \dd \tau.\]
Integrating on $\HH$, we get, via H\" older's inequality, that $\|L^*_{G_\eps}(\psi)\|_{L^2(\HH)}^2$ is bounded from above by:
\[\int_0^1 \int_\HH\left[ \int_{\RR^N}  k_{\widetilde{G}}(|V|_e)|V|_e\rho(\eps |V|_e) \dd V\cdot
\int_{\RR^N} k_{\widetilde{G}}(|V|_e)|V|_e\rho(\eps |V|_e) |\nabla\psi(\exp_y(-\tau\eps y_N V))|^2\dd V\right] \dd \mu(y) \dd \tau.\]
Since $\rho$ is an increasing function, we get for all $\eps\in(0,1)$ the inequality:
\begin{align*}
\int_{\RR^N} k_{\widetilde{G}}(|V|_e)|V|_e\rho(\eps |V|_e) \dd V
= & \Vol(\mathbb{S}^{N-1})\int_0^\infty k_{\widetilde{G}}(r) r \rho(\eps r) r^{N-1} \dd r\\
\leq & \Vol(\mathbb{S}^{N-1})\int_0^\infty k_{\widetilde{G}}(r) r \sinh(r)^{N-1}\dd r \leq M(\widetilde{G}).
\end{align*}
 Therefore,
\[\|L^*_{G_\eps}(\psi)\|_{L^2(\HH)}^2\leq M(\widetilde{G}) \int_0^1\int_{\RR^N}
k_{\widetilde{G}}(|V|_e)|V|_e\rho(\eps |V|_e) \int_\HH |\nabla\psi(\exp_y(-\tau\eps y_N V))|^2 \dmu(y) \dd V \dd \tau. \]
Using again that $\rho$ is increasing, we obtain by using Lemma \ref{lemmaJacobianExpFixedV}, for $\Psi=|\nabla \psi|^2$ and for the vector $-\tau\eps V$, that,  $\forall \eps\in (0,1)$,
\[\|L^*_{G_\eps}(\psi)\|_{L^2(\HH)}^2\leq M(\widetilde{G})\|\nabla \psi\|_{L^2(\HH)}^2 
\int_{\RR^N} k_{\widetilde{G}}(|V|_e)|V|_e\rho(|V|_e)e^{(N-1)|V|_e}\dd V.\] 
Since
\[\int_{\RR^N} k_{\widetilde{G}}(|V|_e)|V|_e\rho(|V|_e)e^{(N-1)|V|_e}\dd V=\Vol(\mathbb{S}^{N-1})\int_0^\infty k_{\widetilde{G}}(r) r \sinh(r)^{N-1} e^{(N-1)r}\dd r\leq M(\widetilde{G}),\]
it follows that $L_{G_\eps}^*:H^1(\HH)\to L^2(\HH)$ is bounded with norm less than $M(\widetilde{G})$.
\end{proof}

\begin{lemma}
\label{lemma.LGEpstoX}
Let $G$ satisfy Hypotheses  \ref{G1.intro} and \ref{G2.intro}. Then, for every $\psi\in H^1(\HH)$,
\begin{equation}
\label{lemma.LGEpstoX.eq}
\lim_{\eps\to 0}\|L^*_{G_\eps}(\psi)-X_{\widetilde{G}}(\psi)\|_{L^2 ( \HH)}= 0. 
\end{equation}
\end{lemma}
\begin{proof}[Proof of Lemma \ref{lemma.LGEpstoX.eq}]
The proof consists of three steps:

\textit{Step 1}: We further assume that $\widetilde{G}$ is $C^1$ and has uniform compact support away from the null vector in each tangent fiber, which means that there exists some $r_0>1$ such that, if $|W|\notin \left[\frac{1}{r_0},r_0\right]$,
\[\widetilde{G}(y,W)=0,\, \forall y\in \HH.\]
This is equivalent to $k_{\widetilde{G}}$ defined in Hypothesis \ref{G2.intro} being compactly supported in $(0,\infty)$.

In view of Lemma \ref{computations.LGEpsAdj}, 
\[
L^*_{G_\eps}(\psi)(y)  = -\int_{\mathcal{A}_{T_y\HH}[r_0^{-1},r_0]} \widetilde{G}\left(y, W\right)
W\cdot \int_0^1\nabla F_y(-\tau\eps W)\dd \tau \rho(\eps |W|) \dmu(W),
\]
where ${\mathcal{A}_{T_y\HH}[r_0^{-1},r_0]}$ stands for the annular domain in $T_y\HH$ centred at the origin, with radius inner radius $\frac{1}{r_0}$ and outer radius $r_0$:
\[{\mathcal{A}_{T_y\HH}[r_0^{-1},r_0]}:=\left\{W\in T_y\HH: |W|\in \left[\frac{1}{r_0},r_0\right]\right\}.\]
Since we work with continuous functions on compact domains, we can apply dominated convergence in the integral above. Using that $\rho(0)=1$,
\[X_{\widetilde{G}}(y)=-\int_{T_y\HH} \widetilde{G}\left(y, W\right)W\dmu(W)\text{ and }
X_{\widetilde{G}}(\psi)(y)=X_{\widetilde{G}}(y)\cdot \nabla F_y(0),\]
we arrive to:
\begin{equation}
\label{LGAdjToX.pointwise}
\lim_{\eps\to 0} L_{G_\eps}^*(\psi)(y)=X_{\widetilde{G}}(\psi)(y), \, \forall y\in \HH.
\end{equation}

Taking into account that $\widetilde{G}(y,W)=0$ for any $y\in \HH$ and any $|W|<r_0$, it follows that $L_{G_\eps}^*(\psi)$ vanishes outside $K^{r_0}:=\{y\in \HH: d(y,K)\leq r_0\}$, where $K$ is the support of $\psi$. This enables us to further apply dominated convergence in \eqref{LGAdjToX.pointwise} to finish the proof of \eqref{lemma.LGEpstoX.eq} in the compactly supported case.\\

\textit{Step 2:} We approximate the function $\widetilde{G}$ with compactly supported functions with respect to the tangent vector, like those considered in Step 1.

More precisely, for every $\eta>0$, we construct a function $\widetilde{G^\eta}:T\HH\to [0,\infty)$
invariant with respect to the geodesic flow such that:
\begin{equation}
\label{GGEta.p2}
\widetilde{G^\eta}\leq \widetilde{G},
\end{equation}
\begin{equation}
\label{GGEta.p3}
k_{\widetilde{G^\eta}}\text{ has compact support in }(0,\infty)
\end{equation}
and
\begin{equation}
\label{GGEta.p4}
M(\widetilde{G}-\widetilde{G^\eta})\leq \eta.
\end{equation}
The construction is done as follows: since $\widetilde{G}$ satisfies Hypothesis \ref{G2.intro}, we can take $R_\eta>0$ such that:
\[\Vol(\mathbb{S}^{N-1})\int_{\RR\setminus [\frac{1}{R_\eta},R_\eta]} k_{\widetilde{G}}(r) (1+r) (e^r \sinh(r))^{N-1} \dd r <\eta .\]
We further consider a smooth non-negative function $\phi_\eta$ which is equal to $1$ on $[\frac{1}{R_\eta},R_\eta]$, less then or equal to $1$ on $\RR$ and vanishes on $[0,\frac{1}{R_\eta+1}]\cup [R_\eta+1,\infty)$. We finally set: \[\widetilde{G^\eta}(x,V):=\widetilde{G}(x,V) \phi_\eta(|V|),\]
so that properties \eqref{GGEta.p2}-\eqref{GGEta.p4} are obviously satisfied.

Let $G^\eta: \HH\times\HH \to [0,\infty)$,
\begin{equation}
\label{GGEta.p1}
G^\eta(x,y)=\widetilde{G^\eta}(x,V_{x,y}).
\end{equation}
Also let $X^\eta$ be the first moment vector field corresponding to $\widetilde{G^\eta}$, which has the form:
\begin{equation}
\label{XEta}
X^\eta(x)=-\int_{T_x\HH} \widetilde{G^\eta}(x,W) W \dmu(W).
\end{equation}
Therefore, Remark \ref{XLeqMG} implies that:
\[\|X-X^\eta\|_{L^\infty(\HH)}\leq M(\widetilde{G}-\widetilde{G^\eta})\leq \eta,\]
which, together with the pointwise estimate:
\begin{equation}
\label{estimate.X.pointwise}
|X(\psi)-X^\eta(\psi)|=|\nabla \psi\cdot (X-X^\eta)| \leq \|X-X^\eta\|_{L^\infty(\HH)} |\nabla \psi|,\,\forall\psi\in H^1(\HH),
\end{equation}
leads to:
\begin{equation}
\label{estimate.X.XEta}
\|X(\psi)-X^\eta(\psi)\|_{L^2(\HH)}\leq \eta\|\nabla \psi\|_{L^2(\HH)}.
\end{equation}
On the other hand, we apply Lemma \ref{H1L2Boundedness} for the function $G-G^\eta$, together with \eqref{GGEta.p4} to obtain:
\begin{equation}
\label{estimate.LG.LGEta}
\|L_{G_\eps}^*(\psi)-L_{G^\eta_\eps}^*(\psi)\|_{L^2(\HH)}\leq \eta\|\nabla\psi\|_{L^2(\HH)}, \forall \eps\in(0,1).
\end{equation}
Moreover, applying the compactly supported case (Step 1) to $G^\eta$, we obtain that, for $\psi\in C_c^1(\HH)$, we have:
\begin{equation}
\label{LGEtaconvergence}
\lim_{\eps\to 0}\|L^*_{G^\eta_\eps}(\psi)-X^\eta(\psi)\|_{L^2(\HH)}= 0.
\end{equation}
The triangle inequality for the $L^2(\HH)$ norms:
\begin{align*}
\|L^*_{G_\eps}(\psi)-&X(\psi)\|_{L^2(\HH)} \\
&\leq \|L^*_{G_\eps}(\psi)-L^*_{G^\eta_\eps}(\psi)\|_{L^2(\HH)}+ \|L^*_{G^\eta_\eps}(\psi)-X^\eta(\psi)\|_{L^2(\HH)}+\|X^\eta(\psi)-X(\psi)\|_{L^2(\HH)},
\end{align*}
together with \eqref{estimate.LG.LGEta}, \eqref{LGEtaconvergence} and \eqref{estimate.X.XEta}, leads to \eqref{lemma.LGEpstoX.eq}.\\

\textit{Step 3}: Estimates \eqref{estimate.X.pointwise} and \eqref{estimate.LG.LGEta} allow us to prove that \eqref{lemma.LGEpstoX.eq} is valid for any $\psi\in H^1(\HH)$.
\end{proof}

We are now able to write the:
\begin{proof}[Proof of Theorem \ref{transportConvergence.intro}]
Theorem \ref{nonLocalExistence} implies that, for every $T>0$, the family $(u^\eps)_{\eps>0}$ is bounded in $L^2([0,T],L^2(\HH))$, so we can extract a subsequence that converges weakly, as $\eps \to 0$, to some $u\in L^2([0,T],L^2(\HH))$.
From Lemma \ref{H1L2Boundedness} and Lemma \ref{lemma.LGEpstoX} we deduce by dominated convergence that, for any $\varphi\in C_c^1([0,T)\times \HH)$,  $L_{G_\eps}^*\varphi \rightarrow X(\varphi)$ in $L^2((0,T),L^2(\HH))$ and then:
\[\begin{aligned}
\lefteqn{\int_0^\infty \int_\HH u(t,x)\partial_t \varphi(t,x) \dmu(x) \dd t }\\
&=\lim_{\eps\to 0}  \int_0^\infty \int_\HH u^\eps(t,x)\partial_t \varphi(t,x) \dmu(x) \dd t\\
&=-\int_\HH u_0(x)\varphi(0,x) \dmu(x) - \lim_{\eps\to 0}  \int_0^\infty \int_\HH \partial_t u^\eps(t,x)\varphi(t,x) \dmu(x) \dd t\\
&= -\int_\HH u_0(x)\varphi(0,x) \dmu(x) - \lim_{\eps\to 0}  \int_0^\infty \int_\HH L_{G_\eps}(u^\eps(t))(x)\varphi(t,x) \dmu(x) \dd t\\
&= -\int_\HH u_0(x)\varphi(0,x) \dmu(x) -  \lim_{\eps\to 0} \int_0^\infty \int_\HH u^\eps(t,x) L_{G_\eps}^*(\varphi(t))(x) \dmu(x) \dd t\\
&=-\int_\HH u_0(x)\varphi(0,x) \dmu(x) - \int_0^\infty \int_\HH u(t,x) X(\varphi(t))(x)\dmu(x) \dd t,
\end{aligned}\]
which means, by Definition \ref{def.weakTranspX} and Theorem \ref{transpXUniqueness}, that $u$ is the unique weak solution of problem \eqref{IntrotranspX}. In conclusion, since every subsequence of the initial sequence $(u_\eps)_{\eps>0}$ admits a subsequence weakly convergent to the same function $u$, then the weak convergence $u_\eps\rightharpoonup u$ is valid for the whole initial family of solutions $(u_\eps)_{\eps>0}$.
\end{proof}

\section{A large class of functions $G$ satisfying the assumptions of Theorem \ref{transportConvergence.intro}}
\label{section.exampleOfG}
In this section, we work on the Poincar\'e ball  model to provide an explicit construction of a very general class of functions $G$ satisfying Hypotheses \ref{G1.intro},  \ref{G2.intro} and \ref{G3.intro}. As outlined in Section \ref{section:geodesicFlow}, we construct a function $g:\mathbb{S}^{N-1}\times \mathbb{S}^{N-1}\times (0,\infty)\to [0,\infty)$, such that, if
\[\widetilde{G}(x,V)=g(\sigma^-(x,V),\sigma^+(x,V),|V|),\]
for any $V\neq 0$ in $T_x\HH$, then Hypotheses \ref{G2.intro} and \ref{G3.intro} are satisfied.
\begin{proposition}
\label{prop:exampleOfG}
If $g:\mathbb{S}^{N-1}\times \mathbb{S}^{N-1}\times (0,\infty)\to [0,\infty)$ has the separated variables expression
\[g(\sigma^-,\sigma^+,r)=g_1(\sigma^-,\sigma^+) \xi(r)\]
such that $g_1:\mathbb{S}^{N-1}\times \mathbb{S}^{N-1}\to [0,\infty)$ and $\xi:(0,\infty)\to [0,\infty)$ are $C^1$ and $\xi$ satisfies:
\begin{equation}
\label{exampleG:integrabilityXi}
\int_0^\infty \xi(r) (1+r)(e^r\sinh(r))^{N-1} dr <\infty,
\end{equation}
then $\widetilde{G}(x,V)=g(\sigma^-(x,V),\sigma^+(x,V),|V|)$ satisfies Hypotheses \ref{G1.intro}, \ref{G2.intro} and \ref{G3.intro}.
\end{proposition}
\begin{proof}
Proposition \ref{GInvariantGFq} implies that $\widetilde{G}$ is invariant under the geodesic flow.

Next, since $\widetilde{G}(x,V)\leq \|g_1\|_\infty\,  \xi(|V|)$ and $g_1$ is continuous on the compact space $\mathbb{S}^{N-1}\times \mathbb{S}^{N-1}$, it is clear that $\widetilde{G}$ satisfies Hypothesis \ref{G2.intro}.

We prove now that the vector field $X_{\widetilde{G}}$ is $C^1$. Indeed, we recall that:
\[X_{\widetilde{G}}(x)=\int_{T_x\HH} \widetilde{G}(x,V) V \dd \mu(V) =\int_{T_x\HH} g_1(\sigma^-(x,V),\sigma^+(x,V)) V \xi(|V|) \dd\mu(V).\]
Writing this in polar coordinates and taking into account that
\[\sigma^-(x,V)=\sigma^-\left(x,\frac{V}{|V|}\right)\text{ and } \sigma^+(x,V)=\sigma^+\left(x,\frac{V}{|V|}\right)\] 
(both vectors $V$ and $\frac{V}{|V|}$ describe the same geodesic), we obtain:
\begin{equation}
\label{XInPolar}
X_{\widetilde{G}}(x)=\int_0^\infty \xi(r)r^{N} dr \int_{T_x^1\HH} g_1(\sigma^-(x,W),\sigma^+(x,W)) W\dd\mu(W),
\end{equation}
where $T_x^1\HH$ stands for the unit sphere in $T_x\HH$.

Since, for a compact neighbourhood $\mathcal{V}_x$ of $x\in \HH$, the set
\[ T^1\mathcal{V}_x=\bigcup_{y\in \mathcal{V}_x} T_y^1\HH\]
is compact and the functions $\sigma^-$ and $\sigma^+$ are smooth on this set, we obtain by \eqref{XInPolar} and the fact that $g_1$ is of class $C^1$, that the vector field $X_{\widetilde{G}}$ is also $C^1$.
\end{proof}
From the proof above, it is clear that we can relax the conditions on the function $g_1$ in the sense that we require it to be bounded, but $C^1$ only outside the diagonal $\{(\sigma,\sigma)\}\subset \mathbb{S}^{N-1}\times \mathbb{S}^{N-1}$.

\begin{remark}
In \cite{MR3646345}, the authors consider a general local problem and construct suitable non-local kernels in order to approximate (in a relaxation result) the solution of the local equation. The approach of our paper is converse, in the sense that we start from kernels which obey some restrictions and arrive to a local problem as the limit of the non-local processes. 
\end{remark}

\begin{remark}
It is an interesting question to describe more precisely the set $\mathcal{X}$ of vector fields $X_G$ that could appear from the relaxation of the kernels $G_\varepsilon$. We remark that this set is quite general. For example, given any finite number of points $x_1,x_2,\ldots x_m\in \HH$ and any tangent vectors $V_i \in T_{x_i}\HH$, we can construct a function $\widetilde{G}$ such that the associated $X_G\in \mathcal{X}$ satisfies $X_G(x_i)=V_i$, for every $i=\overline{1,m}$. To do this -- assuming for simplicity, that the pairs $(\sigma^-(x_i,V_i),\sigma^+(x_i,V_i))$ are distinct -- we use Proposition \ref{prop:exampleOfG}, where we fix the function $\xi$ satisfying the required integrability properties and then appropriately choose $g_1$ supported in non-overlapping neighbourhoods of $(\sigma^-(x_i,V_i),\sigma^+(x_i,V_i))$, in order to obtain $X_G(x_i)=V_i$, for $i=\overline{1,m}$.
\end{remark}

\section{Compactness result on manifolds}
\label{section:compactenessManifolds}
In this section, we state and prove a compactness result for functions defined on a general class of Riemannian manifolds. This compactness tool is the manifold analogue of \cite[Theorem 3.1]{IgnatIgnatStancuDumitru}, which is in turn based on \cite[Theorem 6.11]{rossiNonLocal} and it will allow us to prove the convergence of the solutions of the non-local non-linear convection-diffusion problem \eqref{into.NonlocalConvDiff.eps} to the solution of the local one. The result is contained in the following:

\begin{theorem}
\label{compactnessResultM}
Let $(M,g)$ be a $N$-dimensional complete connected Riemannian manifold and $\Lambda:[0,\infty)\to [0,\infty)$ a continuous function such that $\Lambda(0)> 0$. We denote $\Lambda_\eps(r)=\frac{1}{\eps^N}\Lambda\left(\frac{r}{\eps}\right)$, $\eps>0$.\newline
Let $T>0$ and  $(u^\eps)_{\eps>0}$  a bounded family of functions in $L^2([0,T],L^2(M))$ such that
\begin{equation}
\label{conditionRhoU}
 \eps^{-2}\int_0^T \int_{M\times M} \Lambda_\eps(d(x,y))|u^\eps(t,y)-u^\eps(t,x)|^2 \dmu_g(x)\dmu_g(y) \dd t \leq \GG<\infty, \ \forall \eps>0.
\end{equation}
\begin{enumerate}[label=(\arabic*)]
\item If $u^\eps$ converges weakly to some $u\in L^2([0,T],L^2(M))$, then $u\in L^2([0,T],H^1(M))$ and there exists a constant $C$ which only depends on $T$, $M$ and $\Lambda$, such that
\[\int_0^T \|\nabla_g u(t)\|^2_{L^2(M)} \dd t\leq C\, \GG.\]
\item If, in addition, $ \|\partial_t u^\eps\|_{L^2([0,T],H^{-1}(M))}$ is uniformly bounded in $\eps>0$, then $(u^\eps)_{\eps>0}$ has a subsequence which converges strongly in $L^2([0,T],L^2_\loc(M))$.
\end{enumerate}
\end{theorem}
\begin{remark}
\label{remark.particularLambda}
Since condition \eqref{conditionRhoU} is also satisfied for any $\widetilde{\Lambda}\leq \Lambda$, we can assume during the proofs that $\Lambda$ is smooth, compactly supported and non-increasing, as in the Euclidean version \cite[Theorem 3.1]{IgnatIgnatStancuDumitru}.
\end{remark}
The main idea of the proof of Theorem \ref{compactnessResultM} is to transfer the functions defined on $M$ to subsets of the Euclidean space, for which the conclusion is true by \cite[Theorem 3.1]{IgnatIgnatStancuDumitru}. In this sense, the following chart covering lemma, inspired from \cite[Lemma 3.1]{SobolevNormsManifolds}, essentially flattens the manifold $M$ locally to subsets of $\RR^N$.
\begin{lemma}
\label{chartSmallBends}
Let $E \subseteq M$ a compact set, then for each $\eta \in(0,1)$, there exists a finite family $\left(U_{k}\right)_{k=1}^{Q}, Q=Q(\eta)$, of bounded open sets of $M$ such that:
\begin{enumerate}[label=\roman*)]
\item
 $U_{k} \cap U_{l}=\emptyset, \,\forall k \neq l$ and the closures $(\overline{U}_k)_{k=1}^Q$ cover $E$.
\item for every $k=1,\ldots,Q$, the set $\overline{U}_k$ can be written as the intersection of a countable family $\left(U_{k}^{\tau}\right)_{\tau=1}^\infty$ of sets contained in the domain of a coordinate chart $\left(V_{k}, \phi_{k}\right)$ such that every set $U_{k}^{\tau}$ is a finite union of disjoint smooth bounded domains.
\item for every $k=1,\ldots,Q$, the following properties are satisfied:
\[
\begin{gathered}
(1-\eta)\left|\phi_{k}(x)-\phi_{k}(y)\right|_e \leq d(x, y) \leq(1+\eta)\left|\phi_{k}(x)-\phi_{k}(y)\right|_e, \\
1-\eta \leq \sqrt{\det g_{ij}(x)} \leq 1+\eta,
\end{gathered}
\]
for every $x, y \in V_{k}$. Here, $(g_{ij})_{i,j}$ is the matrix corresponding to the metric tensor in the local chart $(V_k,\phi_k)$.

\item For every $k=1,\ldots,Q$, the operator norm $\left\|\left.d \phi_{k}\right|_{x}\right\|_{(T_x M \to \RR^N)}$ of $\left.d \phi_{k}\right|_{x}:\left(T_{x} M,|\cdot|_{g}\right) \rightarrow\left(\mathbb{R}^N,|\cdot|_e\right)$ is bounded by:
\[
1-\eta \leq\left\|\left.d \phi_{k}\right|_{x}\right\|_{(T_x M \to \RR^N)} \leq 1+\eta,
\]
for every $x \in V_{k}$.
\item the boundary of $U_k$ has zero volume, for every $k=1,\ldots,Q$.
\end{enumerate}
\end{lemma}
\begin{proof}
For every point $x\in E$ we consider $(V_x,\phi_x)$ normal geodesic coordinates around $x$, restricted such that the chart domain $V_x$ is a ball and properties {\em iii)-iv)} above are satisfied. Now, we consider the family $(W_x)_{x\in E}$ of balls concentric with $V_x$, but with half radius and extract the finite subcover $(W_k)_{k=1}^Q$ of $E$.
The family $(U_k)_{k=1}^Q$ is constructed as follows: $U_1=W_1$, $U_2=W_2\setminus\overline{U}_1$, $U_3=W_3\setminus\overline{U_1\cup U_2}$ and so on. We obtain that properties {\em i)} and {\em v)} above are satisfied.

A classical result attributed to Whitney implies that every $\overline{U}_k$  is the set of zeros of a smooth function $\zeta_k:M\to [0,\infty)$, which in turn equals $1$ outside a large compact set. Sard's lemma provides us a sequence $(\sigma_\tau)_{\tau\geq 1}$ converging to zero as $\tau\rightarrow \infty$, of regular values of $\zeta_k$. Considering:
\[U_k^\tau=\zeta_k^{-1}([0,\sigma_\tau]),\]
it follows that $U_k^\tau$ is a smooth compact manifold with boundary $\zeta_k^{-1}(\sigma_\tau)$, so property {\em ii)} is satisfied.  We further remark that $U_k^\tau$ has finitely many connected components.
\end{proof}
We also need the following result related to \cite[Theorem 6.11, p.~128]{rossiNonLocal}, concerning the Euclidean case:
\begin{lemma}
\label{extendFromRossiNonLocal}
Let $T>0$,  $\Omega$ be a open and non-empty set of $\RR^N$ and $\Lambda$ as in Theorem \ref{compactnessResultM}. If there exists a positive constant $\widetilde \GG$ such that  $(f_\eps)_{\eps>0}$ converges weakly to $f$ in $L^2([0,T],L^2(\Omega))$ and satisfies:
\[\eps^{-2}\int_0^T\int_{\Omega\times\Omega} \Lambda_\eps(|y-x|_e)|f_\eps(t,y)-f_\eps(t,x)|^2\dd x \dd y \dd t\leq \widetilde \GG  , \forall \ \eps>0,\]
then $f\in L^2((0,T), H^1(\Omega))$ and there exists a positive constant $C(N,\Lambda)$ such that:
\[\|\nabla_e f\|^2_{L^2([0,T],L^2(\Omega))}\leq C \,\widetilde{\GG}.\]
\end{lemma}
\begin{proof}
We follow the proof of \cite[Theorem 6.11 (1.i)]{rossiNonLocal} and introduce $\overline{f}$, the extension of $f$ by zero outside of $\Omega$. It satisfies
\[ \int_0^T\int_{\RR^N} \int_{\Omega} \Lambda(|z|_e)\rchi_\Omega(x+\eps z)\left|\frac{\overline{f_\eps}(t,x+\eps z)-f_\eps(t,x)}{\eps}\right|^2 \dd x \dd z \dd t\leq  \widetilde{\GG},\]
and, up to a subsequence, 
\[(\Lambda(|z|_e))^\frac{1}{2}\rchi_\Omega(x+\eps z)\frac{\overline{f_\eps}(t,x+\eps z)-f_\eps(t,x)}{\eps}\rightharpoonup (\Lambda(|z|_e))^\frac{1}{2}z\cdot \nabla_e f(t,x), \text{in}\ L^2((0,T)\times \RR^N\times \Omega). \]
 Therefore, using that the strong norm in a Banach space is weakly lower semicontinuous, we obtain:
\[\int_0^T\int_{\RR^N}\int_\Omega \Lambda(|z|_e)|z\cdot \nabla_e f(t,x)|^2 \dd x \dd z\dd t\leq \widetilde{\GG}.\]
Since 
%The integral above is equal to
%$C_1\|\nabla f\|^2_{L^2(\Omega)}$, where
\[C_1:=\int_{\RR^N} \Lambda(|z|_e)|z\cdot \omega|^2 \dd z, \, \omega\in \mathbb{S}^{N-1},\]
 is independent on the choice of the unit vector $\omega\in \mathbb{S}^{N-1}$ the conclusion follows with $C=\frac{1}{C_1}$
%  , it follows that 
% \[\|\nabla f\|^2_{L^2(\Omega)}\leq\frac{\widetilde{\GG}}{C_1}.\qedhere\]
\end{proof}

\begin{proof}[Proof of Theorem \ref{compactnessResultM}]
% We aim to use the time-dependent equivalent of the characterization in \eqref{H1WithDiv}. 
Let us consider a time-dependent smooth vector field $X$ with $\mathrm{supp}(X)\subseteq [0,T]\times E$ for some compact set $E\subset M$. 
In order to conclude that $u\in L^2([0,T],H^1(M))$, we look for a positive constant $C_u$ independent of $X$ such that:
\begin{equation}
\label{uDivXLeqL2norm}
\left|\int_0^T\int_M u \, \dv_g(X) \dmu_g \dd t\right| \leq C_u \|X\|_{L^2([0,T],L^2(M))}.
\end{equation}

For the set $E$ above, we consider the construction in Lemma \ref{chartSmallBends}, for an arbitrary $\eta\in (0,1)$. In this setting, we denote $X_k$ the expression of $X$ in local coordinates given by $(V_k,\phi_k)$ and, by Lemma \ref{chartSmallBends} {\em iii)-iv)}, we obtain:
\begin{equation}
\label{XEuclLEQXManifold}
\left\|X_k\right\|^2_{L^2(\phi_k(U_k^\tau))}=\int_{\phi_k(U_k^\tau)} |X_k(a)|_e^2 \dd a= \int_{U_k^\tau} |d \phi_k(X)(x)|_e^2 \frac{1}{\sqrt{ \det g_{ij}(x)}} \dmu_g(x)\leq \frac{(1+\eta)^2}{1-\eta} \left\|X\right\|^2_{L^2(U_k^\tau)}.
\end{equation}

For every $k=1,\ldots,Q$ and every $\tau\geq 1$, we obtain:
\[ \eps^{-2}\int_0^T  \int_{U_k^\tau\times U_k^\tau} \Lambda_\eps(d(x,y))|u^\eps(y)-u^\eps(x)|^2 \dmu_g(x)\dmu_g(y) \dd t\leq \GG. \]
Next, we transport everything through the chart map $\phi_k$ and get:
\[\begin{aligned}
\eps^{-2}\int_0^T \int_{\phi_k(U_k^\tau)\times \phi_k(U_k^\tau)} &\Lambda_\eps(d(\phi_k^{-1}(a),\phi_k^{-1}(b)))|u^\eps\circ\phi_k^{-1}(b)-u^\eps\circ\phi_k^{-1}(a)|^2 \\
&\quad\quad \times\sqrt{\det g_{ij}(\phi_k^{-1}(a))} \sqrt{\det g_{ij}(\phi_k^{-1}(b))} \dd a \dd b \dd t\leq \GG.
\end{aligned}
\]
From Lemma \ref{chartSmallBends}, since $\Lambda$ is non-increasing (see Remark \ref{remark.particularLambda}), we obtain that, for $u^\eps_k:=u^\eps\circ \phi_k^{-1}$, it holds:
\[\eps^{-2}\int_0^T  \int_{\phi_k(U_k^\tau)\times \phi_k(U_k^\tau)} \Lambda_\eps((1+\eta)|b-a|_e)|u^\eps_k(b)-u^\eps_k(a)|^2 \dd a \dd b \dd t \leq\frac{\GG}{(1-\eta)^2}. \]
Since $\eta\in (0,1)$ we denote $\widetilde{\Lambda}_\eps(r):=\Lambda_\eps(2r)$ and we obtain:
\begin{equation}
\label{estimate.GammaEpsUKEps}
\eps^{-2}\int_0^T \int_{\phi_k(U_k^\tau)\times \phi_k(U_k^\tau)} \widetilde{\Lambda}_\eps(|b-a|_e)|u^\eps_k(t,b)-u^\eps_k(t,a)|^2 \dd a \dd b \dd t \leq\frac{\GG}{(1-\eta)^2}.
\end{equation}
Next, we transfer the weak convergence $u^\eps\rightharpoonup u$ through the chart map $\phi_k$ using the same changes of variables from above and obtain that:
\[u^\eps_k\rightharpoonup u_k:=u\circ \phi_k^{-1} \text{ in } L^2([0,T],\phi_k(U_k^\tau)).\]
 We  use \eqref{estimate.GammaEpsUKEps} and  Lemma \ref{extendFromRossiNonLocal} to obtain that $u_k\in L^2([0,T],H^1(\phi_k(U_k^\tau)))$ and there exists a constant $C=C_{\tilde \Lambda}$  with
\begin{equation}
\label{localEstimateNormNablaK}
\int_0^T \left\|\nabla_e u_k \right\|^2_{L^2(\phi_k(U_k^\tau))}\leq \frac{C \GG}{(1-\eta)^2}.
\end{equation}

Let $(\delta_k)_{k=1}^Q$ a partition of unity subordinated to the cover $(U_k^\tau)_{k=1}^Q$ (see \cite[Lemma 9.3]{brezis}). Then, for almost every $t\in [0,T]$,
\begin{equation}
\label{uDivXLocalSums}
\int_M u(t) \, \dv_g(X) \dmu_g =\sum_{k=1}^Q \int_{U_k^\tau} u(t)\,\dv_g(\delta_k X) \dmu_g.
\end{equation}
Next, to simplify writing, we drop the time dependence. For every $k=1,\ldots,Q$,
\begin{equation}
\label{uDivXLocal2}
\int_{U_k^\tau} u\, \dv_g (\delta_k X) \dmu_g = \int_{\phi_k(U_k^\tau)} u(\phi_k^{-1}(a))\, \dv _g(\delta_k X)(\phi_k^{-1}(a)) \sqrt{\det g_{ij}(\phi_k^{-1}(a))}\hspace{0.1cm} \dd a.
\end{equation}

Now, we denote $\tilde{\delta}_k=\delta_k\circ\phi_k^{-1}$ and recall that $X_k$ is the expression of $X$ in the local coordinates given by $(V_k,\phi_k)$. The expression of the Riemannian divergence \eqref{riemannianDiv} in these coordinates implies that that the right hand side integral in \eqref{uDivXLocal2} is equal to:
\[\int_{\phi_k(U_k^\tau)} u_k\, \dv_e \left(\tilde{\delta}_k X_k \sqrt{\det (g_{ij}\circ\phi_k^{-1})}\right) \dd a.\]

We use that for a.e. $t\in (0,T)$, $u_k(t)\in H^1(U_k^\tau)$ and $\delta_k X_k$ has compact support in $U_k^\tau$ to obtain that:
\[\int_{U_k^\tau} u\, \dv_g (\delta_k X) \dmu_g
=\int_{\phi_k(U_k^\tau)} \tilde{\delta}_k \nabla_e u_k \cdot X_k \sqrt{\det (g_{ij}\circ\phi_k^{-1})}\hspace{0.1cm} \dd a.\]
Next, using the third property in 
 Lemma \ref{chartSmallBends} and the estimate on $X_k$ in \eqref{XEuclLEQXManifold} we get that:
 \begin{equation}
\label{localGradientNormEstimate}
\left|\int_{U_k^\tau} u\, \dv_g (\delta_k X) \dmu_g\right|\leq (1+\eta) \|\nabla_e u_k  \|_{L^2(\phi_k(U_k^\tau))}  \|X_k\|_{L^2(\phi_k(U_k^\tau))}\leq \frac{(1+\eta)^2}{(1-\eta)^{1/2}} \|\nabla_e u_k  \|_{L^2(\phi_k(U_k^\tau))}   \|X\|_{L^2( U_k^\tau)}.
\end{equation}
Integrating with respect to time variable, relations \eqref{localEstimateNormNablaK} and \eqref{localGradientNormEstimate}  imply that:
\[\left|\int_0^T \int_{U_k^\tau} u \, \dv_g(\delta_k X) \dmu_g \dd t\right|\leq \frac{(1+\eta)^2}{(1-\eta)^\frac{3}{2}} \sqrt{C\, \GG} \left\|X\right\|_{L^2([0,T],L^2(U_k^\tau))}.\]
Considering the sum over $k=1,\ldots,Q$, equality \eqref{uDivXLocalSums} implies:
\[\left|\int_0^T \int_M u\, \dv_g X\, \dmu_g \dd t\right| \leq \frac{(1+\eta)^2}{(1-\eta)^\frac{3}{2}} \sqrt{C\, \GG}\sum_{k=1}^Q \left\|X\right\|_{L^2([0,T],L^2(U_k^\tau))}.\]
Since this is true for all $\tau=1,2,\ldots$, we obtain by Lemma \ref{chartSmallBends} {\em v)} that:
\[\left|\int_0^T \int_M u\, \dv_g X\, \dmu_g \dd t\right|\leq \frac{(1+\eta)^2}{(1-\eta)^\frac{3}{2}} \sqrt{C\, \GG}\,\sum_{k=1}^Q \left\|X\right\|_{L^2([0,T],L^2(U_k))}=\frac{(1+\eta)^2}{(1-\eta)^\frac{3}{2}} \sqrt{C\, \GG}\,\left\|X\right\|_{L^2([0,T],L^2(M))}.\]
Now, we take the limit as $\eta$ tends to $0$ and get:
\[\left|\int_0^T \int_M u\, \dv_g X\, \dmu_g \dd t\right| \leq \sqrt{C\, \GG}\,\left\|X\right\|_{L^2([0,T],L^2(M))},\]
so we are in the setting of \eqref{uDivXLeqL2norm}. We deduce that $u\in L^2([0,T],H^1(M))$ with
\[\int_0^T \|\nabla_g u(t)\|^2_{L^2(M)} \dd t\leq C\, \GG.\]
We have proved the first part of Theorem \ref{compactnessResultM}.

To deal with the second part of the theorem, it suffices to prove the strong convergence $u_\eps \to u$ in $L^2([0,T],L^2(U_k^\tau))$ up to a subsequence, for every $k$ and $\tau$. As usual, we transfer the problem on the Euclidean space via $\phi_k$ and apply the known Euclidean result (\cite[Theorem 3.1]{IgnatIgnatStancuDumitru}).

Indeed, let $u^\eps_k$ as above. We will show that there is a constant $\widetilde{K}>0$ such that, for every $\eps>0$ and $\varphi\in C_c^\infty((0,T)\times \phi_k(U_k^\tau))$:
\[\left|\int_0^T \int_{\phi_k(U_k^\tau)} u^\eps_k(t,a) \partial_t \varphi(t,a) \dd a \dd t\right|^2 \leq \widetilde{K} \int_0^T \|\varphi(t)\|^2_{H^1(\phi_k(U_k^\tau))}\dd t,\]
assuming that we know that there exists $K>0$ such that:
\begin{equation}
\label{utL2H-1Equivalence}
\left|\int_0^T \int_{U_k^\tau} u^\eps(t,x) \partial_t \psi(t,x) \dmu_g(x) \dd t\right|^2 \leq K\int_0^T \|\psi(t)\|^2_{H^1(U_k^\tau)}\dd t,
\end{equation}
for every $\eps>0$ and $\psi\in C_c^\infty((0,T)\times U_k^\tau)$. For details about these equivalent characterizations of the norm $\|\partial_t u\|_{L^2([0,T],H^{-1})}$ see \cite[Propositions 23.20 and 23.23]{zeidlerIIA}.

For $\varphi\in C_c^\infty((0,T)\times \phi_k(U_k^\tau))$, the third part of Lemma \ref{chartSmallBends}, together with \eqref{utL2H-1Equivalence} implies that:
\[\begin{aligned}
& \left|\int_0^T \int_{\phi_k(U_k^\tau)} u^\eps_k(t,a) \partial_t \varphi(t,a) \dd a \dd t\right|^2 = \left|\int_0^T \int_{U_k^\tau} u^\eps(t,x) \partial_t\varphi(t,\phi_k(x))\frac{1}{\sqrt{\det g_{ij}(x)}} \dmu_g(x) \dd t\right|^2\\
& \leq \frac{K}{(1-\eta)^2}\int_0^T \int_{U_k^\tau}\left[|\varphi(t,\phi_k(x))|^2+|\nabla_g [\varphi(t,\phi_k(x))]|_g^2\right]\dmu_g(x)\dd t
\end{aligned}
\]
Since, by the fourth part of Lemma \ref{chartSmallBends},
\[|\nabla_g [\varphi(\phi_k(x))]|_g= \|d(\varphi\circ \phi_k)_x\|_{(T_xM\to \RR)}\leq \|d\varphi_{\phi_k(x)}\|_{(\RR^N\to \RR)} \|d\phi_k\|_{(T_x M \to \RR^N)}\leq (1+\eta)|\nabla_e \varphi(\phi_k(x))|_e,\]
we change variables to the Euclidean space and obtain:
\[\left|\int_0^T \int_{\phi_k(U_k^\tau)} u^\eps_k(t,a) \partial_t \varphi(t,a) \dd a \dd t\right|^2\leq \frac{K(1+\eta)^3}{(1-\eta)^2} \int_0^T \|\varphi(t)\|^2_{H^1(\phi_k(U_k^\tau))}\dd t.\]

Therefore, we proved that $\|\partial_t u_k^\eps\|_{L^2\left([0,T],H^{-1}\left(\phi_k\left(U_k^\tau\right)\right)\right)}$
is uniformly bounded in $\eps>0$, and, since the set $\phi_k\left(U_k^\tau\right)$ has finitely many connected components, each of them being a smooth domain, \cite[Theorem 3.1]{IgnatIgnatStancuDumitru} implies that, up to a subsequence, $(u^\eps_k)_{\eps>0}$ converges strongly to $u_k$ in $L^2([0,T],\phi_k(U_k^\tau))$. With the same change of variables that we used extensively in this proof, we get that $(u_\eps)_{\eps>0}$ converges strongly to $u$ in every $L^2([0,T],L^2(U_k^\tau))$, so the latter convergence takes place in $L^2([0,T],L^2_\loc(M))$.
\end{proof}
\section{Local non-linear convection-diffusion on the hyperbolic space}
\label{section:localNonLinearConvDiff}
In this section, we turn our attention towards a local non-linear convection-diffusion problem on the hyperbolic space. The positive constant $A$ represents the diffusivity coefficient, whereas the locally Lipschitz real function $f$, together with the bounded $C^1$ vector field $X$, will act as the non-linear convection term for the particle system.
\begin{equation}
\label{localConvDiff}
\left\{\begin{array}{ll}
\partial_t u(t,x) = A \Delta u(t,x) -\dv(f(u(t)) X)(x), & x\in \HH, t\geq 0;\\
u(0,x)=u_0(x), & x\in \HH.
\end{array}\right.
\end{equation}
\begin{definition}
\label{def.weakLocalConvDiff} Let $u_0\in L^2(\HH)\cap L^\infty(\HH)$.
We call 
\begin{equation}
\label{defLSpace}
u\in \mathcal{L}:=C([0,\infty),L^2(\HH))\cap L^2_\loc ([0,\infty),H^1(\HH))\cap L^\infty([0,\infty),L^\infty(\HH))
\end{equation}
a \textit{weak solution} of \eqref{localConvDiff} if, for every $\varphi\in C_c^1([0,\infty),H^1(\HH))$,
\begin{equation}
\label{localConvDiffWeakest}
\begin{aligned}
& \int_0^\infty \int_\HH u(t,x) \partial_t \varphi(t,x) \dmu(x) \dd t +\int_\HH u_0(x) \varphi(0,x) \dmu(x) \\
&\quad=\int_0^\infty \int_\HH \left[A\nabla u(t,x)\cdot \nabla \varphi(t,x) -f(u(t,x)) X(x)\cdot \nabla \varphi(t,x)\right] \dmu(x) \dd t.
\end{aligned}
\end{equation}
\end{definition}
We prove an equivalent formulation of the definition above:
\begin{proposition}
\label{equivLocalWeakConvDiff}
A function $u$ belonging to the space $\mathcal{L}$ defined in \eqref{defLSpace}
satisfies \eqref{localConvDiffWeakest} if and only if $\partial_t u\in L^2_\loc ([0,\infty),H^{-1}(\HH))$ and $u$ satisfies:
\begin{equation}
\label{localConvDiffWeak}
\left\{\begin{array}{l}\displaystyle
\langle \partial_t u(t),\psi\rangle_{H^{-1}(\HH),H^1(\HH)} + A \int_\HH \nabla u(t)\cdot \nabla \psi\, \dmu(x) = \int_\HH f(u(t))X(x)\cdot \nabla \psi\, \dmu(x),\\[10pt]
\hspace{10cm} \forall \psi\in H^1(\HH),\text{ a.e. }t\geq 0;\\
u(0)=u_0.
\end{array}\right.
\end{equation}
\end{proposition}
\begin{proof} Let $T>0$ fixed.
The direct statement follows by considering $\varphi(t,x)=\eta(t)\psi(x)$, with $\eta\in C_c^1([0,T))$ and $\psi\in H^1(\HH)$. Indeed, we choose $\eta$ such that $\eta(0)=0$ and obtain the first line in \eqref{localConvDiffWeak} by \cite[Definition 1.4.28]{cazenave}.
Next, we remark that, by density, \eqref{localConvDiffWeakest} also takes place when $\eta\in H^1([0,T])$ with $\eta(T)=0$. Choosing
\[\eta(t)=\left\{
\begin{array}{ll}
-\frac{t}{r}+1, & t\in [0,r] \\
0, & \text{otherwise}
\end{array}\right.
\]
we obtain, by letting $r\to 0$, that $u(0)=u_0$.

The converse statement follows by replacing $\psi$ with $\varphi(t)$ in the first line of \eqref{localConvDiffWeak}  and using \cite[Theorem 1.4.35]{cazenave}.
\end{proof}
Next, we prove the uniqueness of weak solutions:
\begin{proposition}
\label{uniquenessLocalWeakConvDiff}
For every $u_0\in L^2(\HH)\cap L^\infty(\HH)$, there exists at most one $u$ satisfying the properties in Proposition \ref{equivLocalWeakConvDiff}.
\end{proposition}
\begin{proof}
Let $T>0$, $u_1, u_2\in \mathcal{L}$ with $\partial_t u_1,\partial_t u_2\in L^2([0,T],H^{-1}(\HH))$
and such that $u_1,u_2$ satisfy \eqref{localConvDiffWeak} with $u_1(0)=u_2(0)=u_0$. Denoting $v=u_1-u_2$, we obtain that, for every $\psi\in H^1(\HH)$,
\begin{equation}
\label{localWeakForV}
\langle \partial_t v(t),\psi\rangle_{H^{-1}(\HH),H^1(\HH)} + A \int_\HH \nabla v(t)\cdot \nabla \psi \,\dmu(x) = \int_\HH [f(u_1(t))-f(u_2(t))]X(x)\cdot \nabla \psi \,\dmu(x).
\end{equation}
Testing the equality above against $\psi=v(t)$, we get (see \cite[Chapter III, Lemma 1.2]{temam}) that:
\[\frac{d}{dt}\int_\HH |v(t)|^2 \dmu(x) +A \int_\HH |\nabla v|^2\dmu(x) = \int_\HH [f(u_1)-f(u_2)]X\cdot \nabla v(t) \dmu(x)\]
Since $u_1,u_2\in L^\infty([0,\infty),L^\infty(\HH))$, $f$ is locally Lipschitz, and $X$ is a bounded vector field, it follows that there exists a constant $L\geq 0$ such that:
\[\frac{d}{dt} \int_\HH |v(t)|^2 \dmu(x)  +A \int_\HH |\nabla v|^2\dmu(x) \leq L  \int_\HH |v|\, |\nabla v| \dmu(x) .\]
The inequality $ab\leq \frac{L}{4A}a^2+ \frac{A}{L}b^2$ applied in the RHS above implies that:
\[\frac{d}{dt} \int_\HH |v(t)|^2 \dmu(x) \leq \frac{L^2}{4A} \int_\HH |v(t)|^2 \dmu(x).\]
By Gronwall's lemma, it follows that $v=0$, so the weak solution is unique.
\end{proof}

\section{Non-local non-linear convection-diffusion on hyperbolic space}
\label{section:nonlocalNonLinearConvDiff}
This section is concerned with the basic analysis of the non-local non-linear convection-diffusion problem \eqref{IntroNonlocalConvDiff}.
\subsection{A priori norm estimates}
We will study problem \eqref{IntroNonlocalConvDiff} for the initial data
\[u_0\in \mathcal{Z}:=L^1(\HH)\cap L^\infty(\HH)\]
and we will look for  solutions $u\in C^1([0,T],\mathcal{Z})$.
Before proving the existence and uniqueness of solutions, we need some a priori estimates on the $L^1$ and $L^\infty$-norms:
\begin{proposition}
\label{L1LinftyNormControl}Assume that $J:[0,\infty)\to [0,\infty)$ is such that $\int_0^\infty J(r)(\sinh(r))^{N-1}\dd r<\infty$ and $G:\HH\times\HH\to [0,\infty)$ satisfies \eqref{condition.integrabilityG} and \eqref{GSymmetricintegral.eq}. Also let $f$ be a non-decreasing, locally Lipschitz real function.

If $T>0$, $u_0\in \mathcal{Z}$ and $u\in C^1([0,T],\mathcal{Z})$ is a solution of \eqref{IntroNonlocalConvDiff}, then $\|u(t)\|_{L^1(\HH)}\leq \|u_0\|_{L^1(\HH)}$ and $\|u(t)\|_{L^\infty(\HH)}\leq \|u_0\|_{L^\infty(\HH)}$, for every $t\in [0,T]$.
\end{proposition}
\begin{proof}
Without loss of generality, we can assume $f(0)=0$, since it makes no difference in \eqref{IntroNonlocalConvDiff} if we consider the function $f-f(0)$ instead of $f$.

To prove the $L^1$-estimate, we notice that:
\[\begin{aligned}\frac{d}{dt} \int_\HH |u(t,x)| \dmu(x)&=\int_{\HH\times\HH} J(d(x,y))(u(t,y)-u(t,x)){\rm sgn}(u(t,x)) \dmu(y) \dmu(x) \\
&\quad+\int_{\HH\times\HH} G(x,y)(f(u(t,y))-f(u(t,x))){\rm sgn}(u(t,x)) \dmu(y) \dmu(x).
\end{aligned}\]
We only prove that the second integral in the right-hand side is non-positive since, in fact, the first integral is a particular case of the second one.
\[\begin{aligned}&\int_{\HH\times\HH} G(x,y)(f(u(t,y))-f(u(t,x))){\rm sgn}(u(t,x)) \dmu(y)\dmu(x)\\
&\quad\leq \int_{\HH\times\HH} G(x,y)|f(u(t,y))| \dmu(y) \dmu(x) - \int_{\HH\times\HH} G(x,y)|f(u(t,x))| \dmu(y) \dmu(x),
\end{aligned}\]
which vanishes by \eqref{GSymmetricintegral.eq}.
We have used that, since $f$ is non-decreasing and $f(0)=0$, then ${\rm sgn}(f(r))={\rm sgn}(r).$ Therefore, we have shown
the $L^1$-norm estimate of the solution.

For the $L^\infty$ part, we generalise the results in \cite[Lemma 3.1]{Ignat-Rossi-nonlocal}.
\begin{lemma}
\label{integralOnTails}
Let $G:\HH\times\HH\to [0,\infty)$ satisfy \eqref{condition.integrabilityG} and \eqref{GSymmetricintegral.eq}. Then, for every $\theta\in L^1(\HH)$ and every $\delta\geq 0$,
\[\int_{\{\theta(x)\geq \delta\}} \int_\HH G(x,y) \theta(y) \dmu(y) \dmu(x) \leq \int_{\{\theta(x)\geq \delta\}} \theta(x) \int_\HH G(x,y) \dmu(y) \dmu(x)\]
and
\[\int_{\{\theta(x)\leq -\delta\}} \int_\HH G(x,y) \theta(y) \dmu(y) \dmu(x) \geq \int_{\{\theta(x)\leq -\delta\}} \theta(x) \int_\HH G(x,y) \dmu(y) \dmu(x).\]
\end{lemma}
\begin{proof}[Proof of the Lemma]
We will only prove the first part, since the second one follows by replacing $\theta$ with $-\theta$. We begin with the case $\delta=0$. Using that $-\theta\,\rchi_{\{\theta\leq 0\}}$ and $G$ are non-negative, a change of variables and Tonelli's theorem implies:
\[\begin{aligned}\int_{\{\theta(x)\geq 0\}} &\int_\HH G(x,y) \theta(y) \dmu(y) \dmu(x)\leq \int_{\{\theta(x)\geq 0\}} \int_{\{\theta(y)\geq 0\}} G(x,y) \theta(y) \dmu(y) \dmu(x)\\
&=\int_{\{\theta(x)\geq 0\}} \theta(x) \int_{\{\theta(y)\geq 0\}} G(y,x) \dmu(y) \dmu(x)\leq \int_{\{\theta(x)\geq 0\}} \theta(x)  \int_\HH G(y,x) \dmu(y) \dmu(x).
\end{aligned}\]
The conclusion in the case $\delta=0$ follows then by Proposition \ref{GSymmetricIntegral}.

For $\delta>0$, since $\theta$ is integrable, the set $\{\theta(\cdot)\geq \delta\}$ has finite measure. Therefore,
\[\begin{aligned}&\int_{\{\theta(x)\geq \delta\}} \int_\HH G(x,y) \theta(y) \dmu(y) \dmu(x)\\
&\quad=\int_{\{\theta(x)-\delta \geq 0\}} \int_\HH G(x,y)(\theta(y)-\delta) \dmu(y) \dmu(x)  + \delta \int_{\{\theta(x)-\delta \geq 0\}} \int_\HH G(x,y) \dmu(y) \dmu(x)\\
&\quad\leq \int_{\{\theta(x)-\delta \geq 0\}} \int_\HH G(x,y)[(\theta-\delta)^+(y)] \dmu(y) \dmu(x)  + \delta \int_{\{\theta(x)-\delta \geq 0\}} \int_\HH G(x,y) \dmu(y) \dmu(x)\\
&\quad\leq \int_{\{(\theta-\delta)^+(x) \geq 0\}} \int_\HH G(x,y)[(\theta-\delta)^+(y)] \dmu(y) \dmu(x) + \delta \int_{\{\theta(x)-\delta \geq 0\}} \int_\HH G(x,y) \dmu(y) \dmu(x), 
\end{aligned}\]
where the positive part function is defined as $(\theta-\delta)^+:=\max\{(\theta-\delta),0\}$ and we notice that the set $\{(\theta-\delta)^+(x) \geq 0\}$ is, in fact, the whole space $\HH$.
Further, we apply the case $\delta=0$ above for the function $(\theta-\delta)^+$, which is in $L^1(\HH)$, and obtain:
\[\begin{aligned}&\int_{\{\theta(x)\geq \delta\}} \int_\HH G(x,y) \theta(y) \dmu(y) \dmu(x) \\
&\quad\leq \int_{\{(\theta-\delta)^+(x) \geq 0\}} [(\theta-\delta)^+(x)] \int_\HH G(x,y)\dmu(y) \dmu(x)  + \delta \int_{\{\theta(x)-\delta \geq 0\}} \int_\HH G(x,y) \dmu(y) \dmu(x)\\
&\quad= \int_{\{\theta(x)-\delta \geq 0\}} (\theta(x)-\delta)\int_\HH G(x,y)\dmu(y) \dmu(x)  + \delta \int_{\{\theta(x)-\delta \geq 0\}} \int_\HH G(x,y) \dmu(y) \dmu(x)
\end{aligned}\]
The conclusion of the lemma follows.
\end{proof}
We return to the proof of Proposition \ref{L1LinftyNormControl} and denote $m:=\|u_0\|_{L^\infty(\HH)}$. Then,
\begin{equation}\label{LInftyyNorm.equality}
\begin{aligned}\frac{d}{dt} \int_\HH (u(t,x)-m)^+ \dmu(x) =&\int_{\HH\times\HH} J(d(x,y))(u(t,y)-u(t,x))\,\rchi_{\{u(t,\cdot)\geq m\}}(x) \dmu(y) \dmu(x) \\
&  \int_{\HH\times\HH} G(x,y)(f(u(t,y))-f(u(t,x)))\, \rchi_{\{u(t,\cdot)\geq m\}}(x) \dmu(y) \dmu(x) .
\end{aligned}
\end{equation}
We clain that the second integral in the RHS is non-negative and the first one can be treated similarly. Indeed, since $f$ is non-decreasing, the second integral becomes:
\[\int_{\{f(u(t,\cdot))\geq f(m)\}} \int_\HH G(x,y)(f(u(t,y))-f(u(t,x))) \dmu(y) \dmu(x).\]
Since $u(t)\in \mathcal{Z}$ and $f$ is locally Lipschitz, it follows that $f(u(t,\cdot))$ is integrable, so we can use Lemma \ref{integralOnTails} to deduce that:
\[\int_{\{f(u(t,\cdot))\geq f(m)\}} \int_\HH G(x,y)(f(u(t,y)) \dmu(y) \dmu(x) \leq \int_{\{f(u(t,\cdot))\geq f(m)\}} \hspace{-0.6cm}f(u(t,x)) \int_\HH G(x,y) \dmu(y) \dmu(x),\]
which  shows our claim. Then, going back to
equation \eqref{LInftyyNorm.equality}, we conclude that, for every $t\in[0,T]$, $u(t)\leq m$ almost everywhere. Similarly, one can prove that $u(t)\geq -m$ a.e. and obtain the desired $L^\infty$-norm estimate for $u(t)$.
\end{proof}

\subsection{Existence and uniqueness} The existence and uniqueness result for the problem \eqref{IntroNonlocalConvDiff} is a classical application of Banach's Contraction Principle:
\begin{theorem}Let $J$, $G$ and $f$ as in Proposition \ref{L1LinftyNormControl}. Then,
for any  $u_0\in \mathcal{Z}=L^1(\HH)\cap L^\infty(\HH)$ there exists a unique solution $u\in C^1([0,\infty),\mathcal{Z})$ of the problem \eqref{IntroNonlocalConvDiff}.
\end{theorem}
\begin{proof}
We define two operators $\widetilde{L}_J,L_{G,f}:\mathcal{Z}\to \mathcal{Z}$,
\begin{equation}\label{def.LJLGf}
\begin{gathered}
\widetilde{L}_J (w)(x):=\int_\HH J(d(x,y))(w(y)-w(x)) \dmu(y),\\
L_{G,f} (w)(x):= L_G(f(w))=\int_\HH G(x,y)(f(w(y))- f(w(x)) \dmu(y).
\end{gathered}\end{equation}
In fact, the first operator is a particular case of the second one, i.e., $\widetilde{L}_J=L_{J(d(\cdot,\cdot)),{\rm id}}$.

With this notations, problem \eqref{IntroNonlocalConvDiff} can be written in the equivalent form:
\begin{equation}
\label{nonLocalConvDiff.withOperator}
\begin{cases}
\partial_t u(t) =\widetilde{L}_J(u(t))+L_{G,f}(u(t)), & t\geq 0,\\
u(0)=u_0. &
\end{cases}
\end{equation}
Similarly to the proof of Theorem \ref{nonLocalExistence}, taking into account that, for every $x\in \HH$,
$$\int_{\HH} J(d(x,y)) \dmu(y)={\rm Vol}(\mathbb{S}^{N-1}) \int_0^\infty J(r) (\sinh(r))^{N-1} \dd r <\infty,$$
one can prove that $\widetilde{L}_J$ is a bounded linear operator on $\mathcal{Z}$.
Thus, denoting $S(t):=e^{t\widetilde{L}_J}$ the semigoup generated by it, we are looking for fixed points in $C([0,T],\mathcal{Z})$ of the 
operator:
\[(\Phi(u))(t)=S(t) u_0 +\int_0^t S(t-s) L_{G,f}(u(s)) ds.\]
By Banach's fixed point theorem in a ball of $C([0,T],\mathcal{Z})$, we obtain a local solution $u\in C([0,T],\mathcal{Z})$, where the time $T>0$ depends on the $L^1$ and $L^\infty$ norms of the initial data. The $C^1([0,T],\mathcal{Z})$ regularity of the solution follows by the boundedness of the two operators $\widetilde{L}_J$ and $L_{G}$.
The global existence is then a consequence of Proposition \ref{L1LinftyNormControl}.
\end{proof}
The following result will be used to prove $L^2$-norm estimates for the solutions of \eqref{IntroNonlocalConvDiff}.
\begin{proposition}
\label{GFuxFuyuxNegative} Let $G:\HH\times\HH\to [0,\infty)$ satisfy \eqref{condition.integrabilityG} and \eqref{GSymmetricintegral.eq}. Also, consider $f(r)=|r|^{q-1}r$, $q\geq 1$. Then,
\[(L_{G,f}(w),w)_{L^2(\HH)}=\int_{\HH\times\HH} G(x,y) (f(w(y))-f(w(x)))w(x) \dmu(y) \dmu(x)\leq 0,\]
for every $w\in \mathcal{Z}=L^1(\HH)\cap L^\infty(\HH)$.
\end{proposition}
\begin{proof}
Since, from Young's inequality $a^qb\leq \frac{q}{q+1}a^{q+1}+\frac{1}{q+1} b^{q+1}$, for any $a,b\geq 0$, we get
\[\begin{aligned}&\int_{\HH\times\HH} G(x,y) \left||w(y)|^{q-1}w(y)w(x)\right| \dmu(y) \dmu(x)\\
&\quad \quad\quad\quad\leq \int_{\HH\times\HH} G(x,y) \left[\frac{q}{q+1}|w(y)|^{q+1}+\frac{1}{q+1} |w(x)|^{q+1} \right] \dmu(y) \dmu(x).
\end{aligned}\]
Since $G$ satisfies \eqref{GSymmetricintegral.eq}, the RHS above is exactly:
\[\int_{\HH\times\HH} G(x,y) |w(x)|^{q+1} \dmu(y) \dmu(x)\]
and thus the desired inequality holds.
\end{proof}

The following consequence is immediate:
\begin{corollary}
\label{L2decay}
Let $J$ and $G$ be as in Proposition \ref{L1LinftyNormControl}, $f(r)=|r|^{q-1}r$, $q\geq 1$ and $u\in C^1([0,\infty),\mathcal{Z})$ a solution of \eqref{IntroNonlocalConvDiff}. Then, the $L^2(\HH)$ norm of $u$ does not increase in time. Moreover, the following energy estimate holds:
\[\|u(T)\|^2_{L^2(\HH)}+\frac{1}{2}\int_0^T\int_{\HH\times\HH} J(d(x,y))(u(t,y)-u(t,x))^2 \dmu(y)\dmu(x)dt \leq \|u_0\|^2_{L^2(\HH)}.\]
\end{corollary}
\begin{proof}
A change of variables and the Fubini-Tonelli theorem implies that:
\[(-\widetilde{L}_{J}u(t),u(t))_{L^2(\HH)}=\frac{1}{2}\int_{\HH\times\HH} J(d(x,y))(u(t,y)-u(t,x))^2 \dmu(y) \dmu(x).\]
Therefore, multiplying \eqref{nonLocalConvDiff.withOperator} by $u(t)$ and integrating on $[0,T]\times \HH$, the conclusion follows by Proposition \ref{GFuxFuyuxNegative}.
\end{proof}
\begin{remark}
The results in this section can be immediately generalised to arbitrary measure spaces $(X,\mathcal{A},\mu)$ by replacing $J$ with a symmetric kernel (i.e. $J(x,y)=J(y,x)$, $\forall x,y\in X$) which satisfies:
$$\sup_{x\in X}\int_{X} J(x,y) \dmu(y) <\infty$$
and imposing on the kernel $G:X\times X\to [0,\infty)$ the constraints \eqref{condition.integrabilityG} and \eqref{GSymmetricintegral.eq}.
\end{remark}

\section{Relaxation limit for convection-diffusion equation}
\label{section:relaxationConvDiff}
This section is dedicated to the proof of Theorem \ref{convergenceConvDiff.intro}. We recall that the equation \eqref{into.NonlocalConvDiff.eps} can be written as:
\begin{equation}
\label{nonLocalConvDiffEps.withOperator}
\begin{cases}
\partial_t u(t) =\widetilde{L}_{J_\eps}(u(t))+L_{G_\eps,f}(u(t)), & t\geq 0\\
u(0)=u_0, &
\end{cases}
\end{equation}
where $J_\eps$ and $G_\eps$ are defined in \eqref{def.JEps} \eqref{defGEps.intro} and the operators $\widetilde{L}_{J_\eps}$ and $L_{G_\eps,f}$ are given in \eqref{def.LJLGf}.

First, we remark that, in the setting of Theorem \ref{convergenceConvDiff.intro}, $\int_0^\infty J_\eps(r)(\sinh(r))^{N-1}{\rm d}r$ is finite and $G_\eps$ satisfies \eqref{condition.integrabilityG} and \eqref{GSymmetricintegral.eq}. Therefore, we can apply the results in the previous section to obtain well-posedness and $L^1$, $L^2$ and $L^\infty$-norm boundedness for the solutions $u^\eps$ of \eqref{nonLocalConvDiffEps.withOperator}.

Before proceeding to the actual proof, we need some results concerning the sequence of operators $(\widetilde{L}_{J_\eps})_{\eps>0}$.

\begin{lemma}
\label{IntJBoundedonH1}
Let $J:[0,\infty)\to [0,\infty)$ such that the quantity $\widetilde{M}(J)$ defined in \eqref{J1.intro.eq} is finite. Then, for every $\psi\in H^1(\HH)$ and $\eps\in (0,1)$,
\[(-\widetilde{L}_{J_\eps}\psi,\psi)_{L^2(\HH)}=\frac{\eps^{-N-2}}2\int_{\HH\times\HH} J\left(\frac{d(x,y)}{\eps}\right)(\psi(y)-\psi(x))^2 \dmu(y) \dmu(x)\leq \widetilde{M}(J) \|\nabla \psi\|_{L^2(\HH)}^2.\]
\end{lemma}
\begin{proof} The equality above follows as in the proof of Corollary \ref{L2decay}. For the inequality, let us denote $I_J^\eps(\psi):=2 (-\widetilde{L}_{J_\eps}\psi,\psi)_{L^2(\HH)}$. 
For every $x\in \HH$, we change the variables $y=\exp_x(W)$ as in the proof of Proposition \ref{H1L2Boundedness} and get:
\[\begin{aligned}I^\eps_J(\psi)&=\eps^{-N-2}\int_\HH \int_{T_x\HH} J\left(\frac{1}{\eps}|W|\right)(\psi(\exp_x(W))-\psi(x))^2 \rho(|W|) \dmu(W) \dmu(x)\\
&=\eps^{-2}\int_\HH \int_{T_x\HH} J\left(|W|\right)(\psi(\exp_x(\eps W))-\psi(x))^2  \rho(\eps|W|) \dmu(W) \dmu(x)\\
&=\int_\HH \int_{T_x\HH} J\left(|W|\right)\left(\int_0^1 \nabla \psi(\exp_x(\tau\eps W))\cdot P(x,\exp_x(\tau\eps W))(W) \dd \tau \right)^2  \rho(\eps|W|) \dmu(W) \dmu(x).
\end{aligned}\]
Since the parallel transport is an isometry, $|P(x,\exp_x(\tau\eps W))(W)|=|W|$ and, since $\rho$ is increasing, $\rho(\eps|W|)\leq \rho(|W|)$. Therefore,
\[I^\eps_J(\psi)\leq \int_0^1 \int_\HH \int_{T_x\HH} J(|W|)|W|^2 |\nabla \psi(\exp_x(\tau\eps W))|^2  \rho(|W|) \dmu(W) \dmu(x) \dd \tau.\]
Working on the half-space model, we change the variables $V=\frac{1}{x_N}W$, so $|V|_e=|W|$ and we obtain:
\[I^\eps_J(\psi)\leq\int_0^1 \int_{\RR^N} J(|V|_e)|V|_e^2 \rho(|V|_e) \int_\HH |\nabla \psi(\exp_x(\tau\eps x_N V))|^2 \dmu(x) \dd V \dd \tau.\]
Now, we apply Lemma \ref{lemmaJacobianExpFixedV} and use that $\rho(r)=\left(\frac{\sinh(r)}{r}\right)^{N-1}$ to get that:
\[I^\eps_J(\psi)\leq \int_{\RR^N} J(|V|_e)|V|_e^2 \rho(|V|_e)\, e^{(N-1)|V|_e} \|\nabla \psi\|_{L^2(\HH)}^2 \dd V\leq \widetilde{M}(J) \|\nabla \psi\|_{L^2(\HH)}^2.\]
\end{proof}

\begin{lemma}
\label{uEpsLJEpsEstimateGradient}
Let $J:[0,\infty)\to [0,\infty)$ as in the previous Lemma, $\eps\in(0,1)$ and $v\in L^2([0,T],L^2(\HH))$ for which there exists a constant $\GG>0$ such that:
\begin{equation}
\label{JEpsEstimateB}
\eps^{-N-2}\int_0^T \int_{\HH\times\HH} J\left(\frac{d(x,y)}{\eps}\right)(v(t,y)-v(t,x))^2  \dmu(y) \dmu(x)\dd t \leq \GG.
\end{equation}
Then, for every $\phi\in L^2([0,T],H^1(\HH))$,
\begin{equation}
\label{uEpsLJEpsEstimateGradient.eq}
\left|\left( v,\widetilde{L}_{J_\eps}(\phi)\right)_{L^2([0,T],L^2(\HH))}\right|\leq \sqrt{\widetilde{M}(J)\frac{\GG}{2}}\quad \|\nabla \phi\|_{L^2([0,T]
,L^2(\HH))}.
\end{equation}
\end{lemma}
\begin{proof}
A change of variables implies, via the Fubini-Tonelli theorem, that:
\begin{equation}
\label{LJEpsBilinear}
\begin{aligned}
\left( v,\widetilde{L}_{J_\eps}(\phi)\right)_{L^2([0,T],L^2(\HH))}=\eps^{-N-2}\int_0^T \int_{\HH\times\HH} J\left(\frac{d(x,y)}{\eps}\right)(\phi(t,y)-\phi(t,x))v(t,x) \dmu(y) \dmu(x) \dd t
\\
=-\frac{1}{2}\eps^{-N-2}\int_0^T\int_{\HH\times\HH} J\left(\frac{d(x,y)}{\eps}\right)(v(t,y)-v(t,x))(\phi(t,y)-\phi(t,x)) \dmu(y) \dmu(x)\dd t.
\end{aligned}
\end{equation}
Further, we apply the Cauchy-Schwarz inequality and use estimate \eqref{JEpsEstimateB} to obtain:
\begin{equation}
\label{JUepsPhiEstimate1}
\left( v,\widetilde{L}_{J_\eps}(\phi)\right)_{L^2([0,T],L^2(\HH))}^2\leq \frac{\GG}{2}\eps^{-N-2}\int_0^T \int_{\HH\times\HH} J\left(\frac{d(x,y)}{\eps}\right)(\phi(t,y)-\phi(t,x))^2 \dmu(y) \dmu(x) \dd t.
\end{equation}
Finally, we plug the result of Lemma \ref{IntJBoundedonH1} in \eqref{JUepsPhiEstimate1} and arrive to the conclusion.
\end{proof}

\begin{lemma}
\label{convergenceLJEpsToLaplacian}
Let $J$ satisfy Hypothesis \ref{J1.intro} and $(u^\eps)_{\eps>0}$ a sequence converging weakly to $u$ in $L^2([0,T],L^2(\HH))$. We further assume that there exists a constant $\GG>0$ such that:
\begin{equation}
\label{JEpsEstimateBUEps}
\eps^{-N-2}\int_0^T \int_{\HH\times\HH} J\left(\frac{d(x,y)}{\eps}\right)(u^\eps(t,y)-u^\eps(t,x))^2  \dmu(y) \dmu(x)\dd t \leq \GG,\, \forall \eps\in (0,1).
\end{equation}
Then $u\in L^2([0,T],H^1(\HH))$ and, for every $\varphi\in L^2([0,T],H^1(\HH))$,
\begin{equation}
\label{uEpsLJToGradient}
\lim_{\eps\to 0}\left( u^\eps,\widetilde{L}_{J_\eps}(\varphi)\right)_{L^2([0,T],L^2(\HH))} = -A_J\int_0^T \int_\HH \nabla u(t,y)\cdot \nabla \varphi(t,y) \dmu(y) \dd t.
\end{equation}
\end{lemma}
\begin{proof}
Since the estimate \eqref{JEpsEstimateBUEps} holds, then the first part of Theorem \ref{compactnessResultM} implies that $u\in L^2([0,T],H^1(\HH))$.
The proof of \eqref{uEpsLJToGradient} consists in three steps:

Step 1: We prove that, for $\psi\in C_c^\infty(\HH)$ and $J$ compactly supported,
\[\lim_{\eps\to 0}\left\|\widetilde{L}_{J_\eps}(\psi)-A_J\Delta \psi\right\|_{L^2(\HH)}= 0,\]
where we recall that by $\Delta$ we understand the Laplace-Beltrami operator on $\HH$. We proceed as in the proof of Theorem \ref{transportConvergence.intro} and denote $F_y:T_y\HH\to \RR$, $F_y(W):=\psi(\exp_y(W))$, the expression of $\psi$ in normal geodesic coordinates around the point $y$. With this notation, by a change of variables we obtain:
\[\widetilde{L}_{J_\eps}(\psi)(y)=\eps^{-2}\int_{T_y\HH} J(|W|)(F_y(\eps W))-F_y(0))\rho(\eps |W|) dW.\]
Next, a Taylor expansion up to order two with integral reminder leads to:
\[F_y(\eps W)-F_y(0)= \eps \nabla F_y(0)\cdot W +\eps^2\int_0^1 (1-\tau) \sum_{i,j=1}^N \partial_{ij} F_y(\tau\eps W)  W_i W_j \dd \tau, \]
where the differentiation $\partial_{ij}$ and the components of $W$ are considered with respect to the same orthonormal basis of $T_y\HH$ we used in the normal coordinates expression of $\psi$ (in particular, for the half-space model, we can use the basis $\left\{y_N\frac{\partial}{\partial y_i}\right\}_{i=1}^N$). The first order term in the Taylor expansion vanishes when we multiply it by the radial function $J(|W|)\rho(\eps|W|)$ and integrate on $T_y\HH$, so we obtain:
\begin{equation}
\label{LJEpsEquality1}
\widetilde{L}_{J_\eps}(\psi)(y)=\int_{T_y\HH} J(|W|)\int_0^1 \sum_{i,j=1}^N\partial_{ij}F_y(\tau\eps W) W_i W_j (1-\tau) d \tau\rho(\eps |W|) \dmu(W).
\end{equation}
Since, for now, we assume that $J$ is compactly supported, we use \eqref{LJEpsEquality1}, together with the fact that $\rho(0)=1$, to obtain by dominated convergence that:
\begin{equation}
\label{LJEps.pointwiseLimit}
\lim_{\eps\to 0} \widetilde{L}_{J_\eps}(\psi)(y)=\int_{T_y\HH} J(|W|) \sum_{i,j=1}^N\partial_{ij}F_y(0) W_i W_j \int_0^1(1-\tau) d \tau \dmu(W),\, \forall y\in \HH.
\end{equation}
Next, a well-known property of normal coordinates on a Riemannian manifold implies that:
\[\Delta \psi(y)=\sum_{i=1}^N \partial_{ii} F_y(0).\]
Therefore, using again that $J(|W|)$ is radial, we obtain:
\begin{equation}
\label{AJLaplPhi}
\begin{aligned}
\int_{T_y\HH} J(|W|)& \sum_{i,j=1}^N\partial_{ij}F_y(0) W_i W_j \int_0^1(1-\tau) d \tau \dmu(W)\\
&=\Delta \psi(y)\frac{1}{2N}\int_{T_y\HH} J(|W|)|W|^2 \dmu(W)=A_J\Delta \psi(y).
\end{aligned}
\end{equation}
Using \eqref{LJEps.pointwiseLimit} and \eqref{AJLaplPhi}, the same dominated convergence argument as in \textit{Step 1} of the proof of Lemma \ref{lemma.LGEpstoX} implies that:
\[\lim_{\eps\to 0}\left\|\widetilde{L}_{J_\eps}(\psi)-A_J\Delta\psi\right\|_{L^2(\HH)}=0,\]
since that $J$ and $\psi$ are compactly supported. The dominated convergence argument is also valid for the time-dependent case, so, for $\varphi\in C_c^\infty((0,T)\times \HH)$ we can use the weak convergence of $(u^\eps)_{\eps>0}$ to $u$ to obtain:
\[\lim_{\eps \to 0}\left( u^\eps,\widetilde{L}_{J_\eps}(\varphi)\right)_{L^2([0,T],\HH)} =A_J\int_0^T \int_\HH u(t,y) \Delta \varphi(t,y) \dmu(y).\]
An integration by parts argument leads to \eqref{uEpsLJToGradient} in the case of $J$ and $\varphi$ compactly supported.

Step 2: We keep the compact support condition on $J$, but we let $\varphi\in L^2([0,T],H^1(\HH))$. Since we know that the convergence \eqref{uEpsLJToGradient} takes place for compactly supported functions, Lemma \ref{uEpsLJEpsEstimateGradient} and the density of $C_c^\infty(\HH)$ in $H^1(\HH)$ imply that the convergence is also valid in this setting.

Step 3: We drop the compact support condition on $J$, replacing it with Hypothesis \ref{J1.intro}. As in the second step in the proof of Lemma \ref{lemma.LGEpstoX}, for every $\eta>0$, we approximate $J$ with a compactly supported function $J^\eta$ such that $J^\eta\leq J$ and $\widetilde{M}(J-J^\eta)\leq \eta$. Therefore,  $(J-J^\eta)$ satisfies the hypothesis of Lemma \ref{uEpsLJEpsEstimateGradient}, which in turn leads to:
\[
\begin{aligned}
\left|\left( u^\eps,\widetilde{L}_{J_\eps}(\varphi)-\widetilde{L}_{J^\eta_\eps}(\varphi)\right)_{L^2([0,T],L^2(\HH))}\right| &\leq \sqrt{\widetilde{M}(J-J^\eta)\frac{\GG}{2}} 
\quad \|\nabla \varphi\|_{L^2([0,T],L^2(\HH))}\\
&\leq  \sqrt{\frac{\eta \GG}{2}} \quad \|\nabla \varphi\|_{L^2([0,T],L^2(\HH))}.
\end{aligned}
\]
The conclusion of the Lemma follows by a triangle inequality argument similar to the one in the proof of Lemma \ref{lemma.LGEpstoX}.
\end{proof}

\begin{proof}[Proof of Theorem \ref{convergenceConvDiff.intro}]
The energy estimate in Corollary \ref{L2decay} reads as follows, for every $T>0$:
\[\|u^\eps(T)\|^2_{L^2(\HH)}+\frac{1}{2}\eps^{-N-2}\int_0^T \int_{\HH\times\HH} J\left(\frac{d(x,y)}{\eps}\right)(u^\eps(t,y)-u^\eps(t,x))^2 \dmu(y) \dmu(x) \dd t\leq \|u_0\|^2_{L^2(\HH)}.\]
This inequality has two consequences:
\begin{equation}
\label{uEpsL2Control}
\|u^\eps(t)\|_{L^2(\HH)}\leq  \|u_0\|_{L^2(\HH)},\, \forall t\geq 0;
\end{equation}
\begin{equation}
\label{JuEpsBounded}
\eps^{-N-2}\int_0^T \int_{\HH\times\HH} J\left(\frac{d(x,y)}{\eps}\right)(u^\eps(t,y)-u^\eps(t,x))^2 \dmu(y) \dmu(x) \dd t\leq 2\int_\HH |u_0|^2\dmu(x),\hspace{0.1cm}\forall\, T>0.
\end{equation}
The inequality \eqref{uEpsL2Control} implies that $(u^\eps)_{\eps>0}$ is bounded in $L^2([0,T],L^2(\HH))$, so, up to a subsequence (which we also denote by $(u^\eps)_{\eps>0}$), it converges weakly to some $u\in L^2([0,T],L^2(\HH))$. Therefore, we are in the setting of Lemma \ref{convergenceLJEpsToLaplacian} with $\GG=2\|u_0\|^2_{L^2(\HH)}$, from which we deduce that $u\in L^2([0,T],H^1(\HH))$. Further, Lemma \ref{uEpsLJEpsEstimateGradient} implies that:
\begin{equation}
\label{uEpsLJEpsEstimateGradientU0}
\left|\left( u^\eps,\widetilde{L}_{J_\eps}(\varphi)\right)_{L^2([0,T],L^2(\HH))}\right|\leq \sqrt{\widetilde{M}(J)}\,\|u_0\|_{L^2(\HH)}\|\nabla \varphi\|_{L^2([0,T],L^2(\HH))}.
\end{equation}

Next,  we study the uniform boundedness of
$\|\partial_t u^\eps\|_{L^2([0,T],H^{-1}(\HH))}$. Indeed, since, by \eqref{LJEpsBilinear}, the operator $\widetilde{L}_{J_\eps}$ is self-adjoint on $L^2(\HH)$, testing \eqref{nonLocalConvDiffEps.withOperator} against an arbitrary $\varphi\in L^2([0,T],H^1(\HH))$ leads to:
\begin{equation}\label{utEpsVarphiEstimate1}
\begin{aligned}
\left|\left( \partial_t u^\eps(t,x), \varphi(t,x)\right)_{L^2([0,T],L^2(\HH))}\right| &\leq \left|\left( u^\eps,\widetilde{L}_{J_\eps}(\varphi)\right)_{L^2([0,T],L^2(\HH))}\right|\\
&\quad + \left|\left( f(u^\eps), L_{G_\eps}^*(\varphi)\right)_{L^2([0,T],L^2(\HH))}\right|.
\end{aligned}
\end{equation}
Further, the Cauchy-Schwarz inequality and Lemma \ref{H1L2Boundedness} imply that:
\begin{equation}\label{fUEpsLGEps.estimate}
\left|\left( f(u^\eps), L_{G_\eps}^*(\varphi)\right)_{L^2([0,T],L^2(\HH))}\right|\leq M(\widetilde{G})\|\nabla \varphi\|_{L^2([0,T],L^2(\HH))}\|f(u^\eps)\|_{L^2([0,T],L^2(\HH))}.
\end{equation}
By Proposition \ref{L1LinftyNormControl}, we obtain that:
\begin{equation}\label{uEpsLInftyBound}
\|u^\eps\|_{L^\infty([0,\infty),L^\infty(\HH))}\leq \|u_0\|_{L^\infty(\HH)}
\end{equation}
and, since, $f(r)=|r|^{q-1}r$, with $q\geq 1$, we have that:
\begin{equation}\label{fUEps.estimate}
\|f(u^\eps)\|_{L^2([0,T],L^2(\HH))} \leq \|u_0\|_{L^\infty}^{q-1} \|u\|_{L^2([0,T],L^2(\HH))}\leq \|u_0\|_{L^\infty}^{q-1}\sqrt{T}\|u_0\|_{L^2(\HH)}.
\end{equation}
Therefore, \eqref{uEpsLJEpsEstimateGradientU0}, \eqref{utEpsVarphiEstimate1}, \eqref{fUEpsLGEps.estimate} and \eqref{fUEps.estimate} lead to: 
\begin{equation}
\label{dtUEpsBound}
\|\partial_t u^\eps\|_{L^2([0,T],H^{-1}(\HH))}\leq \|u_0\|_{L^2(\HH)}\left[\sqrt{\widetilde{M}(J)}+M(\widetilde{G})\|u_0\|_{L^\infty}^{q-1}\sqrt{T}\right].
\end{equation}

As a result, we can apply the second part of Theorem \ref{compactnessResultM} to obtain that up to a subsequence, $(u^\eps)_{\eps>0}$ converges strongly to $u$ in $L^2([0,T],L^2_\loc(\HH))$. Therefore, the continuity of $f$ implies that, up to another subsequence (also denoted by $(u_\eps)_{\eps>0}$), $(f(u^\eps))_{\eps>0}$ converges a.e. in $[0,T]\times \HH$ to $f(u)$.

Next, \eqref{uEpsL2Control} and \eqref{uEpsLInftyBound} imply that the sequence $(f(u^\eps))_{\eps>0}$ is bounded in $L^2([0,T],L^2(\HH))$, which means that, up to a subsequence, it converges weakly in this space. The strong and the pointwise limits obtained above imply that:
\begin{equation}
\label{fUWeakL2}
f(u^\eps)\rightharpoonup f(u)\text{ in }L^2([0,T],L^2(\HH)).
\end{equation}

Until now, we have proven that:
\[u^\eps\rightharpoonup u\text{ and }
f(u^\eps)\rightharpoonup f(u)\text{ weakly in } L^2([0,T],L^2(\HH)).\]
Further, we multiply \eqref{into.NonlocalConvDiff.eps} by a test function $\varphi \in C_c^1([0,T),H^1(\HH))$ and get that:
\begin{equation}
\label{nonlocalConvDiffWeak}
\begin{aligned}
&-\int_0^T \int_\HH u^\eps(t,x)\partial_t \varphi(t,x) \dmu(x) \dd t -\int_\HH u_0(x)\varphi(0,x) \dmu(x) \\
&\quad=\eps^{-N-2}\int_0^T \int_\HH u^\eps(t,y)\int_\HH J\left(\frac{d(x,y)}{\eps}\right)(\varphi(t,x)-\varphi(t,y)) \dmu(x) \dmu(y) \\
 &\quad\quad\quad+\int_0^T \int_\HH f(u^\eps(t,y))\int_\HH G_\eps\left(x,y\right)(\varphi(t,x)-\varphi(t,y)) \dmu(x) \dmu(y)\\
&\quad=\left( u^\eps,\widetilde{L}_{J_\eps}(\varphi)\right)_{L^2([0,T],L^2(\HH))}+ \left(f(u^\eps),L_{G_\eps}^*(\varphi)\right)_{L^2([0,T],L^2(\HH))}.
\end{aligned}
\end{equation}

Lemma \ref{lemma.LGEpstoX} together with \eqref{fUWeakL2} implies that:
\[\lim_{\eps \to 0}\left(f(u^\eps),L_{G_\eps}^*(\varphi)\right)_{L^2([0,T],L^2(\HH)}=\int_0^T \int_\HH f(u(t)) X(\varphi(t)) \dmu(x) \dd t.\]
Therefore, from  Lemma \ref{convergenceLJEpsToLaplacian} and \eqref{nonlocalConvDiffWeak} we deduce that:
\begin{equation}\label{weakLimitUObtained}
\begin{aligned}
& -\int_0^T \int_\HH u(t,x)\partial_t \varphi(t,x)\dmu(x) \dd t -\int_\HH u_0(x)\varphi(0,x) \dmu(x)\\
&\quad = -A_J\int_0^T \int_\HH \nabla u(t,x)\cdot \nabla \varphi(t,x) \dmu(x) \dd t + \int_0^T \int_\HH f(u(t))X(\varphi(t)) \dmu(y) \dd t.
\end{aligned}
\end{equation}

Moreover, passing to the limit in \eqref{dtUEpsBound} (in fact, in the equivalent statement as in \eqref{utL2H-1Equivalence}), we get that $\partial_t u\in L^2([0,T],H^{-1}(\HH))$. We have also proved before that $u\in L^2([0,T],H^1(\HH))$, so a result of Lions and Magenes (see \cite[Chapter III, Lemma 1.2, p.~205]{temam}) implies that:
\[u\in C([0,T],L^2(\HH)).\]
Therefore, also taking \eqref{uEpsLInftyBound} into account, it follows that $u$ is a weak solution of \eqref{localConvDiff.intro}, in the sense of Definition \ref{def.weakLocalConvDiff}. By the uniqueness result in Proposition \ref{uniquenessLocalWeakConvDiff}, the whole initial sequence $(u^\eps)_{\eps>0}$ converges weakly in $L^2([0,T],L^2(\HH))$ and strongly in $L^2([0,T],L^2_\loc(\HH))$ to the unique weak solution of \eqref{localConvDiff.intro}.
\end{proof}

\subsection{Relaxation limit for more general non-linear terms}
\label{sec:generalF}
Some additional restrictions on the kernels $G$ and $J$ allow us, in fact, to prove the main result for more general non-linear terms. More precisely, if we further assume that there exists a constant $c>0$ such that:
\[k_{\widetilde{G}}(r)\leq c\, J(r), \forall r>0,\]
then Theorem \ref{convergenceConvDiff.intro} holds for every non-decreasing locally Lipschitz function $f$.

Indeed, we will prove that the conclusion of Corollary \ref{L2decay} essentially holds and then the rest of the proof of Theorem \ref{convergenceConvDiff.intro} can be straightforwardly adapted. Let us consider $u^\eps\in C^1([0,\infty),\mathcal{Z})$ the solution of \eqref{into.NonlocalConvDiff.eps} and let $a$ be a positive constant that will be chosen later depending on the initial data $u_0\in \mathcal{Z}$. As in Corollary \ref{L2decay}, we get:
\begin{align}
\notag
\frac{d}{dt}\left[ e^{-at} \|u^\eps(t)\|^2_{L^2(\HH)}\right]&= -a\|u^\eps(t)\|^2_{L^2(\HH)} +(\widetilde{L}_{J_\eps}u(t),u(t))_{L^2(\HH)} + (L_{G_\eps,f}(w),w)_{L^2(\HH)}\\ 
\label{eq:expr-D-formal}
 &= -a\|u^\eps(t)\|^2_{L^2(\HH)} -\frac{1}{2}\int_{\HH\times\HH} J_\eps(d(x,y))(u^\eps(t,y)-u^\eps(t,x))^2 \dmu(y) \dmu(x) \\ \notag
& \quad
+ \int_{\HH\times\HH} G_\eps(x,y) (f(u^\eps(t,y))-f(u^\eps(t,x)))u(t,x) \dmu(y) \dmu(x)
\end{align}

Let $M:= \|u_0\|_{L^\infty(\HH)}$. By Proposition \ref{L1LinftyNormControl}, $u(t,x)\in [-M,M]$ for every $t>0$ and almost every $x\in \HH$. Since $f$ is Lipschitz on $[-M,M]$, we obtain that for a.e $x,y\in \HH$,
\[|f(u^\eps(t,y))-f(u^\eps(t,x))|\leq C_M |u^\eps(t,y)-u^\eps(t,x)|, \]
where $C_M$ is the Lipschitz constant of $f$ on $[-M,M]$. Further, this leads to:
\[\begin{aligned}
(L_{G_\eps,f}(u^\eps(t)),u^\eps(t))_{L^2(\HH)} &\leq C_M \int_{\HH\times\HH} G_{\eps}(x,y) |(u^\eps(t,y)-u^\eps(t,x))u^\eps(t,x)| \dmu(y) \dmu(x)\\
&\leq C_M\int_{\HH\times\HH} k_{\widetilde{G}_\eps}(d(x,y)) |(u^\eps(t,y)-u^\eps(t,x))u^\eps(t,x)| \dmu(y) \dmu(x).
\end{aligned}\]
Using the AM-GM inequality and the fact that $k_{\widetilde{G}_\eps}(r)= \eps^{-N-1}k_{\widetilde{G}}\left(\frac{r}{\eps}\right)$, one obtains:
\begin{equation}
\label{eq:estimate-LG-general-f}
\begin{aligned}
(L_{G_\eps,f}(u^\eps(t)),u^\eps(t)) &\leq \frac{1}{4c} \eps^{-N-2} \int_{\HH\times\HH} k_{\widetilde{G}}\left(\frac{d(x,y)}{\eps}\right) (u^\eps(t,y)-u^\eps(t,x))^2 \dmu(y) \dmu(x)\\
&\quad + c\,  C_M^2 \int_{\HH} (u^\eps(t,x))^2 \int_{\HH} \eps^{-N}k_{\widetilde{G}}\left(\frac{d(x,y)}{\eps}\right) \dmu(y) \dmu(x)
\end{aligned}
\end{equation}
Using Lemma \ref{diffExp}, we obtain:
\[\int_{\HH} \eps^{-N}k_{\widetilde{G}}\left(\frac{d(x,y)}{\eps}\right) \dmu(y)={\rm Vol}(\mathbb{S}^{N-1}) \int_0^\infty \eps^{-N} k_{\widetilde{G}}\left(\frac{r}{\eps}\right) (\sinh(r))^{N-1} {\rm d}r,\]
which is at most $M(\widetilde{G})$ by the change of variables $r=\eps s$ and the inequality $\sinh(\eps s)\leq \eps \sinh(s)$, valid for $s\geq 0$ and $\eps\in [0,1]$.

In the end, by choosing $a=c\, C_M^2  M(\widetilde{G})$, from \eqref{eq:expr-D-formal} and \eqref{eq:estimate-LG-general-f} we obtain:
\[\frac{d}{dt}\left[ e^{-at} \|u^\eps(t)\|^2_{L^2(\HH)}\right]\leq -\frac{1}{4}\int_{\HH\times\HH} J_\eps(d(x,y))(u^\eps(t,y)-u^\eps(t,x))^2 \dmu(y) \dmu(x), \]
which leads to:
\[e^{aT}\|u^\eps(T)\|^2_{L^2(\HH)}+\frac{1}{4}\int_0^T\int_{\HH\times\HH} J_{\eps}(d(x,y))(u^\eps(t,y)-u^\eps(t,x))^2 \dmu(y)\dmu(x)dt \leq \|u_0\|^2_{L^2(\HH)},\]
for any $T>0$. In particular, we can apply Theorem \ref{compactnessResultM} to the sequence $(u_\eps)_{\eps>0}$ to obtain the relaxation limit in Theorem \ref{convergenceConvDiff.intro}. We would like to thank the anonymous referee for suggesting the possibility of using more general non-linear terms in this relaxation result.

\subsection*{Financial support}
M.d.M. Gonz\'alez  acknowledges financial support from the Spanish  Government, grant number PID2020-113596GB-I00; additionally, Grant RED2018-102650-T funded by MCIN/AEI/ 10.13039/501100011033, and the ``Severo Ochoa Programme for Centers of Excellence in R\&D'' (CEX2019-000904-S).

L.~I.~Ignat received funding from Romanian Ministry of Research, Innovation and Digitization, CNCS - UEFISCDI, project number PN-III-P1-1.1-TE-2021-1539, within PNCDI III. D.~Manea and S.~Moroianu were partially supported from the UEFISCDI project PN-III-P4-ID-PCE-2020-0794 and the PNRR project III-C9-2023-I8-CF149.

L.~I.~Ignat and D. Manea would like to acknowledge
networking support from the COST Action CA18232  MAT-DYN-NET, supported by COST (European Cooperation in Science and Technology).

\subsection*{Disclosure statement}
The authors report there are no competing interests to declare.

\providecommand{\bysame}{\leavevmode\hbox to3em{\hrulefill}\thinspace}
\providecommand{\MR}{\relax\ifhmode\unskip\space\fi MR }

\providecommand{\MRhref}[2]{%
  \href{http://www.ams.org/mathscinet-getitem?mr=#1}{#2}
}
\providecommand{\href}[2]{#2}

\end{document}